\documentclass{amsart}
\title[Non-commutative virtual structure sheaves]{Non-commutative 
virtual structure sheaves}

\date{}
\author{Yukinobu Toda}

\usepackage{fancybox} 
\usepackage{amscd}
\usepackage{amsmath}
\usepackage{amssymb}
\usepackage{amsthm}
\usepackage{float}
\usepackage[dvips]{graphicx}

\usepackage[all,ps,dvips]{xy}

\usepackage{array}
\usepackage{amscd}
\usepackage[all]{xy}

\DeclareFontFamily{U}{rsfs}{%
\skewchar\font127}
\DeclareFontShape{U}{rsfs}{m}{n}{%
<-6>rsfs5<6-8.5>rsfs7<8.5->rsfs10}{}
\DeclareSymbolFont{rsfs}{U}{rsfs}{m}{n}
\DeclareSymbolFontAlphabet
{\mathrsfs}{rsfs}
\DeclareRobustCommand*\rsfs{%
\@fontswitch\relax\mathrsfs}

\theoremstyle{plain}
\newtheorem{thm}{Theorem}[section]
\newtheorem{prop}[thm]{Proposition}
\newtheorem{lem}[thm]{Lemma}

\newtheorem{defi}[thm]{Definition}
\newtheorem{rmk}[thm]{Remark}
\newtheorem{cor}[thm]{Corollary}

\newtheorem{prop-defi}[thm]{Proposition-Definition}
\newtheorem{thm-defi}[thm]{Theorem-Definition}
\newtheorem{lem-defi}[thm]{Lemma-Definition}

\newtheorem{exam}[thm]{Example}

\newdimen\argwidth
\def\db[#1\db]{
 \setbox0=\hbox{$#1$}\argwidth=\wd0
 \setbox0=\hbox{$\left[\box0\right]$}
  \advance\argwidth by -\wd0
 \left[\kern.3\argwidth\box0 \kern.3\argwidth\right]}

\newcommand{\eE}{\mathcal{E}}
\newcommand{\fF}{\mathcal{F}}

\newcommand{\hH}{\mathcal{H}}
\newcommand{\iI}{\mathcal{I}}
\newcommand{\jJ}{\mathcal{J}}

\newcommand{\lL}{\mathcal{L}}
\newcommand{\mM}{\mathcal{M}}
\newcommand{\nN}{\mathcal{N}}
\newcommand{\oO}{\mathcal{O}}
\newcommand{\pP}{\mathcal{P}}
\newcommand{\qQ}{\mathcal{Q}}

\newcommand{\sS}{\mathcal{S}}

\newcommand{\uU}{\mathcal{U}}
\newcommand{\vV}{\mathcal{V}}
\newcommand{\wW}{\mathcal{W}}

\newcommand{\Hom}{\mathop{\rm Hom}\nolimits}

\newcommand{\dR}{\mathbf{R}}
\newcommand{\dL}{\mathbf{L}}

\newcommand{\id}{\textrm{id}}

\newcommand{\Ext}{\mathop{\rm Ext}\nolimits}
\newcommand{\Spec}{\mathop{\rm Spec}\nolimits}
\newcommand{\Spf}{\mathop{\rm Spf}\nolimits}

\newcommand{\Coh}{\mathop{\rm Coh}\nolimits}

\newcommand{\cneq}{\mathrel{\raise.095ex\hbox{:}\mkern-4.2mu=}}
\newcommand{\eqcn}{\mathrel{=\mkern-4.5mu\raise.095ex\hbox{:}}}

\newcommand{\gr}{\mathop{\rm gr}\nolimits}

\newcommand{\Cone}{\mathop{\rm Cone}\nolimits}

\newcommand{\modu}{\mathop{\rm mod}\nolimits}

\newcommand{\End}{\mathop{\rm End}\nolimits}

\newcommand{\GL}{\mathop{\rm GL}\nolimits}

\newcommand{\tr}{\mathop{\rm tr}\nolimits}

\newcommand{\cl}{\mathop{\rm cl}\nolimits}

\newcommand{\n}{\mathrm{nc}}

\makeatletter
 
  \@addtoreset{equation}{subsection}
\makeatother

\begin{document}
\maketitle

\begin{abstract}
The moduli spaces of 
stable sheaves on projective schemes
admit 
certain
gluing data of Kapranov's NC structures, 
which we call quasi NC structures. 
The 
formal completion of the quasi NC structure 
 at a closed point 
coincides with the pro-representable hull
of the non-commutative deformation functor of the 
corresponding sheaf. 
In this paper, 
we show the existence of 
smooth
non-commutative dg-resolutions 
of the above quasi NC structures,  
and call them  
quasi NCDG structures. 
When there 
are no higher obstruction spaces, 
the quasi NCDG structures
define the notion of 
NC virtual structure sheaves, 
the non-commutative analogue of virtual 
structure sheaves. 
We show that 
the NC virtual structure sheaves are described in terms of 
usual virtual structure sheaves together with 
Schur complexes of the perfect obstruction theories. 
\end{abstract}

\section{Introduction}
The purpose of this paper is to introduce the notion of 
\textit{non-commutative virtual structure sheaves} on the 
moduli spaces of stable sheaves on projective schemes 
without higher obstruction spaces. 
The motivation introducing this concept is 
to construct non-commutative analogue 
of the enumerative invariants 
of sheaves, 
e.g. Donaldson-Thomas (DT) invariants~\cite{Thom}, 
involving non-commutative
deformations of sheaves. 
In this introduction, we first recall some background
of commutative virtual structure sheaves
via smooth commutative dg-schemes, and 
also 
explain quasi 
NC structures on the 
moduli spaces of stable 
sheaves obtained in~\cite{Todnc}. 
We then state 
the existence of
smooth non-commutative dg-enhancements
on the moduli spaces of stable sheaves, 
called quasi NCDG structures, 
which govern the commutative
dg-enhancements and the quasi NC structures.  
The construction of quasi NCDG structures leads to 
the definition of NC virtual structure sheaves. 
  
\subsection{Commutative virtual structure sheaves}
Let $X$
be a projective scheme,
and $M_{\alpha}$ the moduli space of 
stable sheaves on $X$ with Hilbert polynomial $\alpha$. 
In general, the moduli space
$M_{\alpha}$ 
may not have the expected dimension
at $[E] \in M_{\alpha}$:
\begin{align}\label{exp.dim}
\mathrm{exp.dim}_{[E]}M_{\alpha}
\cneq 
\dim \Ext^1(E, E)-\dim \Ext^2(E, E). 
\end{align}
Here $\Ext^1(E, E)$ is the tangent space of $M_{\alpha}$
at $[E] \in M_{\alpha}$, and $\Ext^2(E, E)$ is the obstruction space. 
As long as 
there are no higher obstruction spaces, i.e. 
\begin{align}\label{vanish:higher}
\Ext^{\ge 3}(E, E)=0
\ 
\mbox{ for any }
[E] \in M_{\alpha}
\end{align} the 
expected dimension (\ref{exp.dim})
is 
locally constant on $M_{\alpha}$, and 
in this case
the virtual fundamental class
$[M_{\alpha}]^{\rm{vir}} \in A_{\ast}(M_{\alpha})$
with the expected dimension 
(\ref{exp.dim}) can be constructed 
using the notion of perfect obstruction theory~\cite{BF}. 
The integration of the virtual fundamental class
yields interesting enumerative invariants of sheaves, such as
DT invariants~\cite{Thom}. 

From the construction of the virtual class, it admits a 
natural K-theoretic enhancement (cf.~\cite[Remark~5.4]{BF})
\begin{align}\label{intro:K}
\oO_{M_{\alpha}}^{\rm{vir}} \in K_0(M_{\alpha})
\end{align}
called the 
\textit{virtual structure sheaf} of $M_{\alpha}$. 
It recovers the virtual fundamental class 
by applying the cycle map to (\ref{intro:K}).
It was also suggested by Kontsevich~\cite{Ktor}
that $M_{\alpha}$ may be obtained as 
the zero-th truncation 
of a smooth commutative dg-scheme
$(N_{\alpha}, \oO_{N_{\alpha}, \bullet})$, 
i.e. $\oO_{M_{\alpha}}=\hH_0(\oO_{N_{\alpha}, \bullet})$, 
and the virtual structure sheaf (\ref{intro:K}) 
may be described as 
\begin{align}\label{intro:Kid}
\oO_{M_{\alpha}}^{\rm{vir}}
=\sum_{i\in \mathbb{Z}} (-1)^i [\hH_i(\oO_{N_{\alpha}, \bullet})]. 
\end{align}
The dg-scheme 
$(N_{\alpha}, \oO_{N_{\alpha}, \bullet})$ was
later
constructed by 
To$\ddot{\textrm{e}}$n-Vaqui{\'e}~\cite{Toen2}
and
Behrend-Fontanine-Hwang-Rose~\cite{BFHR}.
Also the identity (\ref{intro:Kid})
was established by Fontanine-Kapranov~\cite{CFK}. 

\subsection{Quasi NC structures on $M_{\alpha}$}
In the previous paper~\cite{Todnc}, the moduli space 
$M_{\alpha}$ turned out to admit
a certain non-commutative structure, 
giving an enhancement of $M_{\alpha}$
different 
from the commutative dg-enhancement
 $(N_{\alpha}, \oO_{N_{\alpha}, \bullet})$. 
Such a non-commutative structure was formulated 
in terms of Kapranov's NC schemes~\cite{Kap17}, which are 
ringed spaces whose structure sheaves are possibly non-commutative, 
but formal in the non-commutative direction. 
We refer to~\cite{PoTu}, \cite{Hendr} for the recent developments 
on Kapranov's NC schemes. 

The above non-commutative structure can be 
naturally observed from the formal deformation theory. 
For $[E] \in M_{\alpha}$, 
the formal deformation theory of $E$ is governed by 
the dg-algebra $\dR \Hom(E, E)$, which 
is quasi-isomorphic to 
a minimal $A_{\infty}$-algebra
\begin{align}\label{intro:A}
(\Ext^{\ast}(E, E), \{m_n\}_{n\ge 2}). 
\end{align}
The formal solution of the Mauer-Cartan 
equation of the $A_{\infty}$-algebra (\ref{intro:A})
yields
the not necessary commutative algebra
\begin{align}\label{intro:R}
R_E^{\n} \cneq \frac{\widehat{T}(\Ext^1(E, E)^{\vee})}
{\left(\sum_{n\ge 2}m_n^{\vee} \right)}.
\end{align}
Here $m_n^{\vee}$ is the 
dual of the $A_{\infty}$-product
\begin{align*}
m_n \colon 
\Ext^1(E, E)^{\otimes n} \to 
\Ext^2(E, E). 
\end{align*}
The
algebra (\ref{intro:R}) 
is an enhancement of 
the commutative algebra $\widehat{\oO}_{M_{\alpha}, [E]}$
in the sense that 
\begin{align}\label{intro:RM}
(R_E^{\n})^{ab} \cong \widehat{\oO}_{M_{\alpha}, [E]}. 
\end{align}
Indeed, the algebra $R_E^{\n}$ is a
pro-representable hull of the non-commutative deformation functor 
of $E$ developed in~\cite{Lau}, \cite{Erik}, \cite{ESe}, \cite{ELO}, 
\cite{ELO2}, \cite{ELO3}. 
The main result of~\cite{Todnc} was 
to construct a kind of globalization of the isomorphism (\ref{intro:RM}), 
as follows:
\begin{thm}\emph{(\cite[Theorem~1.2]{Todnc})}
There exists an affine open cover
$\{V_i\}_{i\in \mathbb{I}}$
of $M_{\alpha}$, 
ringed spaces $V_i^{\n}$ and isomorphisms
$\phi_{ij}$
\begin{align}\label{intro:qnc}
V_i^{\n} = (V_i, \oO_{V_i}^{\n}), \ 
\phi_{ij} \colon V_j^{\n}|_{V_{ij}} \stackrel{\cong}{\to}
V_i^{\n}|_{V_{ij}}
\end{align}
where $V_i^{\n}$ is Kapranov's 
affine NC scheme~\cite{Kap17}, 
such that $\phi_{ij}^{ab}=\id$
and 
$\widehat{\oO}_{V_i, [E]}^{\n} \cong R_E^{\n}$
for any $[E] \in V_i$. 
\end{thm}
The data (\ref{intro:qnc})
was called a \textit{quasi NC structure} of $M_{\alpha}$
in~\cite{Todnc}.

\subsection{Quasi NCDG structures}
Now we have obtained 
two kinds of enhancement of $M_{\alpha}$: 
a commutative dg structure 
and a quasi NC structure. 
It is a natural 
question to construct a
further enhancement which 
governs
these two structures. 
We introduce the notion of 
\textit{quasi NCDG structures}
on commutative dg-schemes, and answer 
this question. 
Roughly speaking, 
a quasi NCDG structure on a commutative 
dg-scheme $(N, \oO_{N, \bullet})$
is an affine open cover 
$\{U_i\}_{i\in \mathbb{I}}$ of $N$
together with 
sheaves of non-commutative dg-algebras 
$\oO_{U_i, \bullet}^{\n}$ on $U_i$ satisfying the 
following: 
\begin{itemize}
\item We have $(\oO_{U_i, \bullet}^{\n})^{ab}=\oO_{N, \bullet}|_{U_i}$. 
\item We have the isomorphisms
$\phi_{ij, \bullet} \colon 
(U_{ij}, \oO_{U_j, \bullet}^{\n}|_{U_{ij}})
\stackrel{\cong}{\to}(U_{ij}, \oO_{U_i, \bullet}^{\n}|_{U_{ij}})$
satisfying $\phi_{ij, \bullet}^{ab}=\id$. 
\end{itemize}
We will show the following result: 
\begin{thm}\emph{(Theorem~\ref{thm:ncvir2})}\label{intro:NCDG}
There is a smooth quasi NCDG structure on 
the dg-moduli space
$(N_{\alpha}, \oO_{N_{\alpha}, \bullet})$
whose zero-th
truncation gives a quasi NC structure (\ref{intro:qnc}). 
\end{thm}
The quasi NCDG structure in Theorem~\ref{intro:NCDG}
fits into the upper half of the 
picture in Figure~\ref{fig:intro}.

\begin{figure}\label{fig:intro}
\caption{Relations of DG structures, quasi NC structures and quasi NCDG
structures}
\begin{align}\notag
\xymatrix{
\fbox{$\begin{array}{c}
\mbox{Quasi NC structure} \\
\{(V_i, \oO_{V_i}^{\n})\}_{i\in \mathbb{I}} 
\end{array}$} 
\ar[rr]^-{\rm{abelization}} & &
\ovalbox{$\begin{array}{c}
\mbox{Moduli space of stable sheaves} \\
M_{\alpha}
\end{array}$} \\
\doublebox{$\begin{array}{c}\mbox{Quasi NCDG structure} \\
\{(U_i, \oO_{U_i, \bullet}^{\n})\}_{i\in \mathbb{I}}
\end{array}$} 
\ar[u]^{\rm{truncation}} 
\ar[d]_{\rm{cohomology}}
\ar[rr]^-{\rm{abelization}} & &
\ovalbox{$\begin{array}{c}
\mbox{Commutative DG structure} \\
(N_{\alpha}, \oO_{N_{\alpha}, \bullet}) 
\end{array}$} \ar[u]^{\rm{truncation}} 
\ar[d]_{\rm{cohomology}} \\
\doublebox{$\begin{array}{c}
\mbox{NC virtual structure sheaves} \\
(\oO_{M_{\alpha}}^{\rm{ncvir}})^{\le d} \in K_0(M_{\alpha})
\end{array}$
} & & 
\ovalbox{$\begin{array}{c}
\mbox{Virtual structure sheaf} \\
\oO_{M_{\alpha}}^{\rm{vir}} \in K_0(M_{\alpha})
\end{array}$}
}
\end{align}
\end{figure}
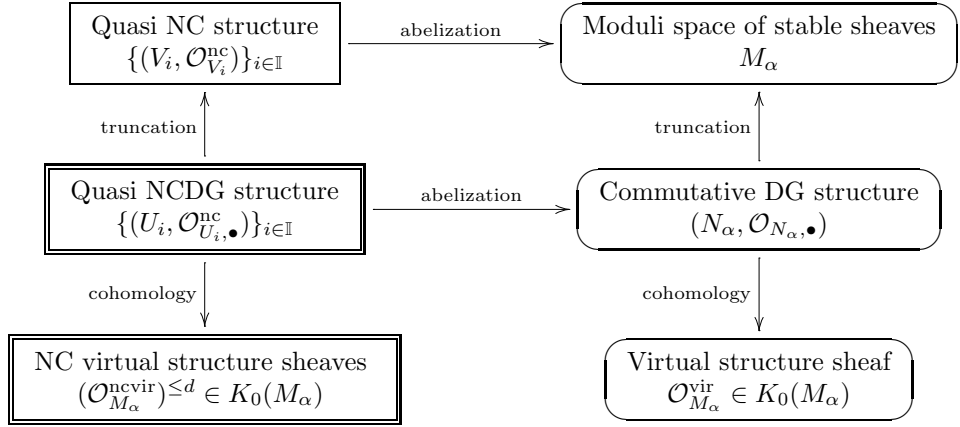

\subsection{NC virtual structure sheaves}
The quasi NCDG structure in Theorem~\ref{intro:NCDG}
is interpreted as a smooth 
dg-resolution of the quasi NC structure (\ref{intro:qnc}).
For simplicity, let us assume that 
the quasi NC structure 
$\oO_{U_i, \bullet}^{\n}$
in Theorem~\ref{intro:NCDG} glue to give 
a global sheaf of non-commutative dg-algebras
$\oO_{N_{\alpha}, \bullet}^{\n}$
on $N_{\alpha}$, i.e. 
the isomorphisms $\phi_{ij, \bullet}$ satisfy the 
cocycle condition. 
As an analogy of the identity (\ref{intro:Kid}),
one may try to define the NC virtual structure sheaf 
by 
\begin{align}\label{intro:sum}
\oO_{M_{\alpha}}^{\rm{ncvir}} =
\sum_{i\in \mathbb{Z}}(-1)^i [\hH_i(\oO_{N_{\alpha}, \bullet}^{\n})]. 
\end{align} 
The issue of the above construction is that 
the sum (\ref{intro:sum})
may be an infinite sum, so 
does not make sense, even if the 
condition (\ref{vanish:higher}) is satisfied. 

Instead of (\ref{intro:sum}), 
if the condition (\ref{vanish:higher}) is satisfied, 
the following sum turns out to be finite 
for each $d\in \mathbb{Z}_{\ge 0}$: 
\begin{align}\label{intro:dvir}
(\oO_{M_{\alpha}}^{\rm{ncvir}})^{\le d}
\cneq \sum_{i\in \mathbb{Z}}
(-1)^i [\hH_i((\oO_{N_{\alpha}, \bullet}^{\n})^{\le d})].
\end{align}
Here 
$(\oO_{N_{\alpha}, \bullet}^{\n})^{\le d}$
is the quotient of $\oO_{N_{\alpha}, \bullet}^{\n}$
by its $d$-th step 
NC filtration
(cf.~Subsection~\ref{subsec:NCfilt}). 
The quotient $(\oO_{N_{\alpha}, \bullet}^{\n})^{\le d}$
is interpreted as a $d$-smooth dg-resolution of the 
quasi NC structure 
$\{(V_i, (\oO_{V_i}^{\n})^{\le d})\}_{i\in \mathbb{I}}$
 on $M_{\alpha}$, 
hence (\ref{intro:dvir}) is regarded as a $d$-smooth 
thickening of (\ref{intro:K}). 
Moreover the sum (\ref{intro:dvir})
also makes sense even if 
the quasi NCDG structure in Theorem~\ref{intro:NCDG}
does not satisfy the cocycle condition (cf.~Definition~\ref{defi:ncvir}).  
We call the sum (\ref{intro:dvir}) 
as \textit{d-th NC virtual structure sheaf} of 
the quasi NC structure (\ref{intro:qnc}).

\subsection{Descriptions via perfect obstruction theories} 
We will prove that 
the $d$-th NC virtual structure sheaf (\ref{intro:dvir})
is described in terms of 
the usual virtual structure sheaf (\ref{intro:K})
together with 
the perfect obstruction theory $\eE_{\bullet}
 \to \tau_{\ge -1} \dL_{M_{\alpha}}$
induced by the cotangent complex of 
the commutative dg-scheme $(N_{\alpha}, \oO_{N_{\alpha}, \bullet})$. 
We have the following result:
\begin{thm}\emph{(Theorem~\ref{thm:formula})}\label{thm:intro:perf}
We have the following formula
\begin{align}\label{intro:ncvir:formula}
(\oO_{M_{\alpha}}^{\rm{ncvir}})^{\le d}=\oO_{M_{\alpha}}^{\rm{vir}} 
\otimes_{\oO_{M_{\alpha}}}
[S_{\oO_{M_{\alpha}}}L_{\oO_{M_{\alpha}}}^{+}
(\eE_{\bullet})^{\le d}_{\bullet}]. 
\end{align}
\end{thm}
Here $L_{\oO_M}(\eE_{\bullet})$ is the
sheaf of super Lie algebras
in $T_{\oO_M}(\eE_{\bullet})$ generated by 
$\eE_{\bullet}$, and 
$S_{\oO_M}(-)$ is the super symmetric product
over $\oO_M$. 
We refer to Subsection~\ref{subsec:viaperf}
for details of the notation of the RHS of (\ref{intro:ncvir:formula}). 
The formula (\ref{intro:ncvir:formula})
implies that (\ref{intro:dvir})
is described using the perfect obstruction theory, 
without referring to quasi NCDG structures.
Also it is described by 
Schur complexes ${\bf S}_{\lambda}(\eE_{\bullet})$ 
for partitions
$\lambda$.  
For example
in the $d=2$ case, the RHS of (\ref{intro:ncvir:formula})
 is written as
(cf.~Corollary~\ref{cor:virform})
\begin{align*}
\oO_{M_{\alpha}}^{\rm{vir}} \otimes_{\oO_{M_{\alpha}}} \left(1+
{\bf S}_{(1, 1)}(\eE_{\bullet}) + {\bf S}_{(2, 1)}(\eE_{\bullet})
+{\bf S}_{(2, 2)}(\eE_{\bullet}) + {\bf S}_{(1, 1, 1, 1)}(\eE_{\bullet})
 \right). 
\end{align*}
In particular, the formula (\ref{intro:ncvir:formula})
implies that 
$(\oO_{M_{\alpha}}^{\rm{ncvir}})^{\le d}=\oO_{M_{\alpha}}^{\rm{vir}}$ if 
the expected dimension of $M_{\alpha}$ is zero. 
This implies that, 
if $X$ is a Calabi-Yau 3-fold, 
the integrations of NC virtual structure sheaves 
coincide with the usual (commutative)
DT invariants. 
On the other hand, if we consider 
moduli spaces of stable sheaves
 on algebraic surfaces or Fano 3-folds
so that they have the positive expected dimensions, 
the resulting NC virtual structure sheaves are in general 
different from the commutative virtual structure sheaves.
In such cases, integrations of their Chern characters may yield 
interesting enumerative invariants. 

In the next paper~\cite{Todnc3}, 
we will pursue another approach 
in constructing interesting enumerative 
invariants of sheaves involving non-commutative 
structures on the moduli spaces of stable sheaves.  
We can consider motives of 
Hilbert schemes of points on 
a quasi NC structure (\ref{intro:qnc}), 
and construct 
certain enumerative invariants 
by integrating the Behrend functions on them. 
If $X$ is a Calabi-Yau 3-fold,  
using wall-crossing argument, we
can 
relate these invariants 
with generalized DT invariants 
counting semistable sheaves on $X$~\cite{JS}, \cite{K-S}
whose definition involves motivic Hall algebras.  
This would give an intrinsic understanding
of the dimension formula~\cite{TodW}
of Donovan-Wemyss's non-commutative widths~\cite{WM}
for floppable rational curves, 
whose detail will be included in~\cite{To-Hu}.

\subsection{Plan of the paper}
In Section~\ref{sec:NC}, we introduce the notion of 
quasi NCDG structures on commutative dg-schemes, and 
use it to define the NC virtual structure sheaves. 
In Section~\ref{sec:Des}, we 
describe the NC virtual structure 
sheaves via perfect obstruction theory, and 
prove Theorem~\ref{intro:ncvir:formula}. 
In Section~\ref{sec:const}, 
we construct quasi NCDG structures on 
the moduli spaces of representations of a certain quiver. 
In Section~\ref{sec:quasi}, using the result of Section~\ref{sec:const}, 
we prove Theorem~\ref{intro:NCDG}. 

\subsection{Acknowledgement}
The author would like to thank 
Tomoyuki Abe, Will Donovan, Zheng Hua 
and Michael Wemyss
for the discussions related to this paper. 
This work is supported by World Premier 
International Research Center Initiative
(WPI initiative), MEXT, Japan, 
and Grant-in Aid
for Scientific Research grant (No.~26287002)
from the Ministry of Education, Culture,
Sports, Science and Technology, Japan.

\subsection{Notation and convention}
In this paper, an algebra always means an associative, 
not necessary commutative,  
$\mathbb{C}$-algebra. 
The tensor product $\otimes$ is over $\mathbb{C}$ if 
there is no subscript. 
Also all the varieties or schemes
are defined over $\mathbb{C}$.

\section{NC virtual structure sheaves}\label{sec:NC}
In this section, 
we recall 
virtual structure sheaves associated to 
commutative dg-schemes, 
and introduce its NC version. 
We prepare the following convention on the 
super symmetric product. 
For a graded vector space $W_{\bullet}$, 
let $S_n$ acts on $W_{\bullet}^{\otimes n}$
in the super sense, i.e. 
the action of the permutation $(i, j) \in S_n$
is  
\begin{align*}
x_1 \otimes \cdots \otimes x_i \otimes \cdots \otimes
x_j &\otimes \cdots \otimes x_n \\
&\mapsto 
(-1)^{\deg x_i \deg x_j} 
x_1 \otimes \cdots \otimes x_j \otimes \cdots \otimes
x_i\otimes \cdots \otimes x_n
\end{align*}
for homogeneous elements $x_1, \cdots, x_n \in W_{\bullet}$. 
The \textit{super symmetric product} of
$W_{\bullet}$ is defined by 
\begin{align*}
S(W_{\bullet}) \cneq \bigoplus_{n\ge 0}
(W_{\bullet}^{\otimes n})^{S_n}.
\end{align*}
The above super
symmetric product is obviously 
generalized 
for graded vector bundles $\qQ_{\bullet}$ on a scheme $N$, 
and we obtain 
the sheaf of super commutative graded algebras 
$S_{\oO_N}(\qQ_{\bullet})$ on $N$. 
Here a graded algebra $A_{\bullet}$ is called 
\textit{super graded commutative} if 
$a_e \cdot a_{e'}=(-1)^{e e'} a_{e'} \cdot a_e$
for $a_e \in A_e$, $a_{e'} \in A_{e'}$. 

\subsection{Commutative dg-schemes}
Let $(N, \oO_{N, \bullet})$
be a
smooth commutative dg-scheme, i.e. 
$N$ is a smooth scheme and $\oO_{N, \bullet}$ is 
a sheaf of super commutative dg-algebras of the form\footnote{It may be
more natural to use the upper index $\oO_N^{\bullet}$
to denote the grading 
of the sheaf of dg-algebras.
In this paper, we use the lower index
$\oO_{N, \bullet}$ as we will use several 
other upper gradings.}
\begin{align*}
\oO_{N, \bullet}=S_{\oO_N}(\qQ_{-1} \oplus \cdots \oplus \qQ_{-k})
\end{align*}
for vector bundles $\qQ_{i}$ on $N$
located in degree $i$. 
The zero-th cohomology of $\oO_{N, \bullet}$ 
is written as $\oO_N/J$ for the ideal 
sheaf $J \subset \oO_N$, hence 
determines a closed subscheme 
$M \subset N$. 
We write
\begin{align*}
\tau_0(N, \oO_{N, \bullet}) \cneq M
\end{align*} 
and call it the \textit{zero-th truncation} of $(N, \oO_{N, \bullet})$. 

Let $E_{\bullet}$ be a finitely generated dg-$\oO_{N, \bullet}$-module.
We will use two kinds of restrictions of it to $M$
\begin{align}\label{restrict}
E_{\bullet}|_{M} \cneq E_{\bullet} \otimes_{\oO_N} \oO_M, \ 
\overline{E}_{\bullet}|_{M} \cneq E_{\bullet} \otimes_{\oO_{N, \bullet}} \oO_M. \end{align}
By setting $
{\bf O}_{\bullet} \cneq \oO_{N, \bullet}|_{M}$, 
the two restrictions (\ref{restrict})
are related by 
\begin{align*}
\overline{E}_{\bullet}|_{M}=E_{\bullet}|_M/({\bf O}_{\le -1} E_{\bullet}|_{M}).
\end{align*}

Let $\Omega_{N, \bullet}$ be the cotangent complex of 
$(N, \oO_{N, \bullet})$. 
The complex $\overline{\Omega}_{N, \bullet}|_{M}$ is described as 
\begin{align*}
\overline{\Omega}_{N, \bullet}|_{M}
=\left(0 \to \qQ_{-k}|_{M} \to 
\cdots \to \qQ_{-2}|_{M} \stackrel{d_{-1}}{\to}
 \qQ_{-1}|_{M} \stackrel{d_0}{\to} \Omega_N|_{M} \to 0  \right). 
\end{align*}
Let $T_{N, \bullet} \cneq 
\hH om_{\oO_{N, \bullet}}(\Omega_{N, \bullet}, \oO_{N, \bullet})$
be the tangent complex of $(N, \oO_{N, \bullet})$. 
\begin{defi}
A smooth commutative dg-scheme $(N, \oO_{N, \bullet})$ is 
called a $[0, 1]$-manifold if the cohomologies of the complex
$\overline{T}_{N, \bullet}|_{M}$ 
are concentrated on $[0, 1]$. 
\end{defi}
Suppose that $(N, \oO_{N, \bullet})$ is a $[0, 1]$-manifold. 
Then $\overline{T}_{N, \bullet}|_{M}$
is quasi-isomorphic to the complex
\begin{align}\label{T:qis}
0 \to T_{N}|_{M} \stackrel{(d_0)^{\vee}}{\to} K \to 0
\end{align}
where $K$ is the kernel of 
$(d_{-1}|_{M})^{\vee} \colon
 (\qQ_{-1}|_{M})^{\vee} \to (\qQ_{-2}|_{M})^{\vee}$, 
which is a locally free sheaf on $N$. 
We define $\eE_{\bullet}$ to be 
the dual of the complex (\ref{T:qis})
\begin{align}\label{E:qis}
\eE_{\bullet} \cneq (0 \to 
K^{\vee} \stackrel{\overline{d}_0}
\to \Omega_{N}|_{M} \to 0)
\end{align}
where $\Omega_{N}|_{M}$
is located in degree zero
and $\overline{d}_0$
is induced by $d_0$. 
We have the quasi-isomorphism
$\overline{\Omega}_{N, \bullet}|_{M}
\stackrel{\sim}{\to}  \eE_{\bullet}$, and 
by~\cite[Proposition~3.2.4]{CFK}, 
we have the morphism of complexes
\begin{align}\label{perfect}
\eE_{\bullet} \to (0 \to J/J^2 \to \Omega_{N}|_{M} \to 0)
\end{align}
giving a perfect obstruction theory on $M$
in the sense of Behrend-Fantechi~\cite{BF}.

\subsection{Commutative virtual structure sheaves}
For a $[0, 1]$-manifold
$(N, \oO_{N, \bullet})$, let us recall the 
virtual fundamental class 
and its K-theoretic enhancement associated 
to the perfect obstruction theory (\ref{perfect}).
Let $C_{M/N}$ be the normal cone of $M$ in $N$
defined by
\begin{align*}
C_{M/N} \cneq \Spec_{\oO_M} \bigoplus_{k\ge 0} J^k /J^{k+1}. 
\end{align*} 
By~\cite[Proposition~1.2.1]{CFK}, 
we have the closed embedding 
$C_{M/N} \subset K$, hence obtain the following diagram
\begin{align*}
\xymatrix{
C_{M/N} \ar@{^{(}->}[rr] \ar[rd] & & K \ar@<0.5ex>[ld] \\
& M \ar@<0.5ex>[ur]^-j. &
}
\end{align*}
Here $j$ is the zero section. 
The virtual fundamental class associated to (\ref{perfect}) is defined by
\begin{align}\label{def:Ktheory}
[M]^{\rm{vir}} \cneq j^{!}[C_{M/N}] \in A_{\bullet}(M). 
\end{align}
By the above definition, the virtual 
fundamental class has a K-theoretic enhancement, 
called the \textit{virtual 
structure sheaf}
\begin{align}\label{Kth:vir}
\oO_M^{\rm{vir}} \cneq [\dL j^{\ast} \oO_{C_{M/N}}] \in K_0(M). 
\end{align}
Note that the virtual 
structure sheaf
recovers the virtual fundamental class
by
\begin{align*}
\cl(\oO_{M}^{\rm{vir}})=[M]^{\rm{vir}} \in A_{\bullet}(M). 
\end{align*}
Here $\cl$ is the cycle map. 

On the other hand, 
note that each cohomology sheaf 
$\hH_i(\oO_{N, \bullet})$ is a coherent $\oO_M$-module, 
which vanishes for $i\ll 0$ by~\cite[Theorem~2.2.2]{CFK}.
Hence for a finitely generated 
dg-$\oO_{N, \bullet}$-module
$E_{\bullet}$, 
each cohomology $\hH_i(E_{\bullet})$ is a coherent 
$\oO_M$-module, and vanishes for $\lvert i \rvert \gg 0$. 
Therefore 
 the following definition makes sense:
\begin{align*}
[\hspace{-0.5mm}[E_{\bullet}]\hspace{-0.5mm}] 
\cneq \sum_{i \in \mathbb{Z}}(-1)^i [\hH_i(E_{\bullet})] \in K_0(M). 
\end{align*}
The above K-theory classes from 
the dg-schemes
are related to the virtual structure sheaf 
as follows: 
\begin{thm}\emph{(\cite[Theorem~4.4.2]{CFK})}\label{thm:vir}
Let $E_{\bullet}$ be a finitely generated locally free 
dg-$\oO_{N, \bullet}$-module. 
Then we have the equality in $K_0(M)$:
\begin{align*}
[\hspace{-0.5mm}[E_{\bullet}]\hspace{-0.5mm}]=
\oO_M^{\rm{vir}} \otimes_{\oO_M} [\overline{E}_{\bullet}|_{M}]. 
\end{align*}
In particular, we have the identity
$\oO_M^{\rm{vir}}=[\hspace{-0.5mm}[\oO_{N, \bullet}]\hspace{-0.5mm}]$
in $K_0(M)$. 
\end{thm}
Note that
for a (not necessary differential) 
graded $\oO_{N, \bullet}$-module
$F_{\bullet}$, 
we can similarly 
define the 
restriction 
$\overline{F}_{\bullet}|_{M}=F_{\bullet} \otimes_{\oO_{N, \bullet}} \oO_M$. 
Let 
\begin{align}\label{K:grade}
[F_{\bullet}]
\in K_0(\oO_{N, \bullet})
\end{align}
be its class in the K-group of 
finitely generated graded $\oO_{N, \bullet}$-modules. 
We have the following corollary of
Theorem~\ref{thm:vir}: 
\begin{cor}\label{cor:K0}
In the situation of Theorem~\ref{thm:vir}, 
let $F_{\bullet}$ be a 
locally free graded $\oO_{N, \bullet}$-module. 
Suppose that 
$[F_{\bullet}]=
[E_{\bullet}]$
in $K_0(\oO_{N, \bullet})$. 
Then we have the identity
\begin{align*}
[\hspace{-0.5mm}[E_{\bullet}]\hspace{-0.5mm}]=
\oO_M^{\rm{vir}} \otimes_{\oO_M} [\overline{F}_{\bullet}|_{M}].
\end{align*}
\end{cor}
\begin{proof}
The corollary follows from 
Theorem~\ref{thm:vir}
since  
$[\overline{E}_{\bullet}|_{M}] \in K_0(M)$
is independent of the differential on $E_{\bullet}$. 
\end{proof}

\subsection{Graded NC filtrations}\label{subsec:NCfilt}
We introduce the non-commutative version of 
some notions recalled in the previous subsections. 
Let $R$ be an 
algebra which is not necessary commutative, 
and $W_{\bullet}$ a finite dimensional 
graded vector space
with $W_i=0$ for $i\ge 0$. 
Below, we call a grading induced 
by the grading on $W_{\bullet}$
as $|_{\bullet}$-\textit{grading}.
We set
the $|_{\bullet}$-graded algebra $\Lambda_{\bullet}$ to be
\begin{align}\label{Lambda:bull}
\Lambda_{\bullet} \cneq R \ast T(W_{\bullet}). 
\end{align}
Here $T(W_{\bullet})$ is the tensor algebra
\begin{align*}
T(W_{\bullet}) \cneq 
\bigoplus_{n\ge 0} W_{\bullet}^{\otimes n}
\end{align*} 
and 
$\ast$ is the free product as $\mathbb{C}$-algebras. 
Note that 
\begin{align*}
\Lambda_{>0}=0, \ 
\Lambda_0=\Lambda, \ 
\Lambda_{-1}=R \otimes W_{-1} \otimes R
\end{align*}
and so on. 
We regard $\Lambda_{\bullet}$ as a 
$|_{\bullet}$-graded super Lie algebra 
by setting 
\begin{align*}
[x, y] \cneq xy-(-1)^{ab} yx, \ 
x \in \Lambda_a, \ y \in \Lambda_b.
\end{align*} 
The subspace
 $\Lambda_{\bullet, k}^{\rm{Lie}} \subset \Lambda_{\bullet}$
is defined to be spanned by the elements of the form
\begin{align*}
[x_1, [x_2, \cdots, [x_{k-1}, x_k]\cdots ]]
\end{align*}
for $x_i \in \Lambda_{\bullet}$, $1\le i\le k$. 
The $|_{\bullet}$-\textit{graded NC filtration} of $\Lambda_{\bullet}$ is the decreasing 
filtration
\begin{align}\label{NCfilt}
\Lambda_{\bullet}=F^0\Lambda_{\bullet} \supset F^1 \Lambda_{\bullet}
 \supset \cdots
\supset F^{d} \Lambda_{\bullet} \supset \cdots
\end{align}
where $F^d \Lambda_{\bullet}$ is the two-sided $|_{\bullet}$-graded 
ideal of $\Lambda_{\bullet}$
defined by
\begin{align*}
F^d \Lambda_{\bullet} \cneq 
\sum_{m\ge 0}\sum_{i_1+\cdots+i_m=m+d}\Lambda_{\bullet}
 \cdot 
\Lambda_{\bullet, i_1}^{\rm{Lie}} \cdot \Lambda_{\bullet} \cdot 
\cdots \cdot \Lambda_{\bullet, i_m}^{\rm{Lie}} \cdot \Lambda_{\bullet}. 
\end{align*}
Note that
$\Lambda_{\bullet}/F^1 \Lambda_{\bullet}$ is the abelization 
$\Lambda^{ab}_{\bullet}$ of $\Lambda_{\bullet}$,
which is a super commutative $|_{\bullet}$-graded algebra written as 
\begin{align}\label{abeliz}
\Lambda^{ab}_{\bullet}=R^{ab} \otimes S(W_{\bullet}). 
\end{align}
We set $\Lambda^{\le d}_{\bullet} \cneq \Lambda_{\bullet}/F^{d+1} 
\Lambda_{\bullet}$, 
and $N^{\le d}_{\bullet} \cneq \Lambda^{\le d}_{\bullet} 
\otimes_{\Lambda_{\bullet}}N_{\bullet}$
for a
graded left $\Lambda_{\bullet}$-module $N_{\bullet}$.  
By the definition of the filtration (\ref{NCfilt}), 
the subquotient 
\begin{align}\label{grFlam}
\gr_{F}(\Lambda_{\bullet}) \cneq 
\bigoplus_{d\ge 0} 
F^d \Lambda_{\bullet}/F^{d+1} \Lambda_{\bullet}
\end{align}
is a bi-graded algebra: 
it is a direct sum of 
\begin{align*}
\gr_F(\Lambda_{\bullet})_{e}^{d}
\cneq \left(F^d \Lambda_{\bullet}/F^{d+1} \Lambda_{\bullet}\right)_{e}
\end{align*} where 
$e$ is the $|_{\bullet}$-grading, and 
we call 
the degree $d$ as $|^{\bullet}$-\textit{grading}.

\subsection{Quasi NC structures}
In the notation of the previous subsection, 
suppose that  
$W_{\bullet}=0$ so that 
$\Lambda_{\bullet}=R$ holds. 
We recall some notions on NC algebras following~\cite{Kap17}. 
\begin{defi}
(i)
An algebra $R$
is called NC nilpotent of degree $d$
\emph{(}resp.~NC nilpotent\emph{)}
if $F^{d+1}R=0$
\emph{(}resp.~$F^n R=0$ for $n\gg 0$\emph{)}.  

(ii) An algebra $R$ is called \textit{NC complete}
if the following natural map is an isomorphism
\begin{align*}
R \to R_{[\hspace{-0.5mm}[ab]\hspace{-0.5mm}]}
 \cneq \lim_{\longleftarrow} R^{\le d}. 
\end{align*}
(iii) An NC nilpotent algebra $R$ is 
called of finite type if $R^{ab}$ is a 
finitely generated $\mathbb{C}$-algebra and 
each $\gr_F^d(R)$ is a finitely 
generated $R^{ab}$-module. 
\end{defi}
Let $R$ be an NC complete algebra. 
For any multiplicative set 
$S \subset R^{ab}$ without zero divisor, 
its pull-back
by the natural surjection $R^{\le d} \twoheadrightarrow
 R^{ab}$
determines the multiplicative set 
in $R^{\le d}$, which 
satisfies the Ore localization condition (cf.~\cite[Proposition~2.1.5]{Kap17}).
In particular, 
one can define the localization 
$S^{-1}R^{\le d}$
of $R^{\le d}$ by $S$. 
 Therefore, similarly to the case of usual affine 
schemes, the NC nilpotent algebra
$R^{\le d}$ determines the sheaf of 
algebras $\widetilde{R}^{\le d}$
on $\Spec R^{ab}$
(cf.~\cite[Definition~2.2.2]{Kap17}). 
Namely, 
the topological basis 
of $\Spec R^{ab}$
is given by 
\begin{align*}
U_f \cneq \{\mathfrak{p} \in \Spec R^{ab} : 
f \notin \mathfrak{p}\}
\end{align*}
and 
$\widetilde{R}^{\le d}$ is the sheafication of the 
presheaf
$U_f \mapsto (f)^{-1} R^{\le d}$, 
where $(f)$ is the multiplicative 
set $\{f^{n} : n\ge 0\}$ in $R^{ab}$. 
Similarly, for any left $R^{\le d}$-module 
$P$, the sheaf $\widetilde{P}$ is 
defined to be the sheafication of the presheaf
\begin{align*}
U_f \mapsto (f)^{-1} R^{\le d} \otimes_{R^{\le d}}
P. 
\end{align*}
The ringed space
\begin{align*}
\Spf R \cneq 
(\Spec R^{ab}, \widetilde{R}), \ 
\widetilde{R} \cneq 
\lim_{\longleftarrow}\widetilde{R}^{\le d}
\end{align*}
is called 
an \textit{affine NC scheme}, or 
an \textit{affine NC structure} 
on the affine scheme $\Spec R^{ab}$. 
The sheaf $\widetilde{R}$ is determined by the localization
for 
a multiplicative set $S \subset R^{ab}$, 
given by 
\begin{align*}
S^{-1} R \cneq \lim_{\longleftarrow}
S^{-1}R^{\le d}. 
\end{align*}
\begin{defi}\emph{(\cite{Kap17})}
A ringed space is called an NC scheme if it is 
locally isomorphic to affine NC schemes. 
\end{defi}
For a scheme $M$, an \textit{NC structure}
is an NC scheme $(M, \oO_M^{\n})$ with 
$(\oO_M^{\n})^{ab}=\oO_M$. 
In~\cite{Todnc}, a weaker 
notion of the NC structures was considered: 
\begin{defi}\emph{(\cite{Todnc})}\label{def:QNC}
Let $M$ be a commutative scheme. 
A quasi NC structure on $M$
consists of an affine 
open cover $\{V_i\}_{i\in \mathbb{I}}$
of $M$, 
affine NC structures
$V_i^{\n}=(V_{i}, \oO_{V_i}^{\rm{nc}})$
on $V_i$
for each $i\in \mathbb{I}$, 
and isomorphisms of NC schemes
\begin{align}\notag
\phi_{ij} \colon 
(V_{ij}, \oO_{V_j}^{\rm{nc}}|_{V_{ij}})
\stackrel{\cong}{\to}
(V_{ij}, \oO_{V_i}^{\rm{nc}}|_{V_{ij}})
\end{align}
satisfying $\phi_{ij}^{ab}=\id$. 
\end{defi}
Let $\nN_{d}$ be the category of NC nilpotent algebras
of degree $d$, and $\nN$ the category of NC nilpotent algebras. 
An exact sequence 
\begin{align}\label{central}
0 \to J \to R_1 \to R_2 \to 0
\end{align}
in $\nN$
is called a \textit{central extension} if $J^2=0$ and 
$J$ lies in the center of $R_1$. 
\begin{defi}
(i) 
An NC nilpotent algebra $R$ 
of degree $d$ is called $d$-smooth if 
it is of finite type and the functor
\begin{align*}
h_{R} \cneq \Hom(R, -) \colon \nN \to \sS et
\end{align*}
is formally $d$-smooth, i.e. 
for any central extension (\ref{central}) in $\nN_d$, 
the map $h_{R}(R_1) \to h_{R}(R_2)$ is surjective. 

(ii) An NC complete algebra $R$ is called smooth if 
$R^{\le d}$ is $d$-smooth for any $d\ge 0$. 
\end{defi}
A quasi NC structure in Definition~\ref{def:QNC} is called smooth 
if each $U_i^{\n}$ is written as 
$U_i^{\n}=\Spf R_i$ for 
a smooth algebra $R_i$. 
If $M$ admits a smooth quasi NC structure, then $M$ must be smooth. 
Conversely, any smooth variety admits a smooth quasi 
NC structure by~\cite[Theorem~1.6.1]{Kap17}.

\subsection{Quasi NCDG structures}
Let 
$R$ be an NC complete algebra and 
$\Lambda_{\bullet}$
a graded algebra 
given by (\ref{Lambda:bull}). 
Suppose that there is a degree one differential
\begin{align*}
Q \colon \Lambda_{\bullet} \to \Lambda_{\bullet +1}
\end{align*}
giving a dg-algebra structure on $\Lambda_{\bullet}$. 
By the Leibniz rule, 
the differential $Q$ preserves the 
filtration (\ref{NCfilt}), 
hence we have the induced
dg-algebra 
$(\Lambda_{\bullet}^{\le d}, Q^{\le d})$. 
Note that each $|_{\bullet}$-degree term 
of $\Lambda_{\bullet}^{\le d}$ is a left
$R^{\le d}$-module, 
and $Q^{\le d}$ is a left 
$R^{\le d}$-module homomorphism. 
Hence
we have the associated 
sheaf of dg-algebras
$\widetilde{\Lambda}_{\bullet}^{\le d}$ on 
$\Spec R^{ab}$, 
which is a complex of 
quasi-coherent left $\widetilde{R}^{\le d}$-modules, 
and the dg-ringed space
\begin{align*}
\Spf \Lambda_{\bullet}^{\le d} \cneq 
(\Spec R^{ab}, \widetilde{\Lambda}_{\bullet}^{\le d}). 
\end{align*}
We see that the above dg-ringed space is 
also locally written 
of the above form. 
\begin{lem}\label{lem:mult}
For any multiplicative set 
$S \subset R^{ab}$ without zero divisor, we have
the canonical isomorphism
\begin{align}\label{id:d}
S^{-1}R^{\le d} \otimes_{R^{\le d}} \Lambda_{\bullet}^{\le d}
\stackrel{\cong}{\to}
\left(S^{-1}R \ast T(W_{\bullet}) \right)^{\le d}. 
\end{align}
\end{lem}
\begin{proof}
Note that there exists a canonical morphism from the LHS 
to the RHS of (\ref{id:d})
by the universality of the localization. 
We prove the isomorphism (\ref{id:d}) by the induction of $d$. 
For $d=0$, the claim 
is obvious since both sides coincide with 
$S^{-1} R^{ab} \otimes S(W_{\bullet})$. 
Suppose that the isomorphism (\ref{id:d}) holds for 
$d\ge 0$. 
Since $S^{-1}R^{\le d+1}$ is a flat 
right $R^{\le d+1}$-module, we have the exact sequence
\begin{align}\notag
0 \to S^{-1} R^{ab} \otimes_{R^{ab}}
\gr_F(\Lambda_{\bullet})^d \to
S^{-1} R^{\le d+1} &\otimes_{R^{\le d+1}}
\Lambda_{\bullet}^{\le d+1} \\
\label{exact:lam}
&\to 
S^{-1}R^{\le d}\otimes_{R^{\le d}}
\Lambda_{\bullet}^{\le d} \to 0. 
\end{align}
On the other hand, we have the 
exact sequence
\begin{align}\notag
0 \to \gr_F(S^{-1}R \ast T(W_{\bullet}))^{d} 
\to&  \left( S^{-1}R \ast T(W_{\bullet}) \right)^{\le d+1} \\
\label{exact:lam2}
&\qquad \qquad
\to \left( S^{-1}R \ast T(W_{\bullet}) \right)^{\le d} \to 0. 
\end{align}
By the assumption of the induction, 
(\ref{exact:lam}), (\ref{exact:lam2}), and the five lemma, 
it is enough to show the isomorphism
\begin{align*}
S^{-1} R^{ab} \otimes_{R^{ab}}
\gr_F(\Lambda_{\bullet})^d 
\stackrel{\cong}{\to}\gr_F(S^{-1}R \ast T(W_{\bullet}))^{d}. 
\end{align*}
In Subsection~\ref{subsec:via}, we will 
see that 
the subquotients of the NC 
filtrations
are
described by  
Poisson envelopes. 
Using Lemma~\ref{lem:isom:Poi} 
in Subsection~\ref{subsec:via},
it is enough to show the isomorphism
\begin{align*}
S^{-1} R^{ab} \otimes_{R^{ab}} 
P(R^{ab} \otimes S(W_{\bullet}))^{d}
\stackrel{\cong}{\to}
P(S^{-1}R^{ab} \otimes S(W_{\bullet}))^{d}. 
\end{align*} 
The above isomorphism follows since 
taking the Poisson envelope commutes with 
the localization (cf.~\cite[Section~4.1]{Kap17}). 
\end{proof}
We define the following dg-ringed 
space
\begin{align}\label{affine:ncdg}
\Spf \Lambda_{\bullet} \cneq 
(\Spec R^{ab}, 
\widetilde{\Lambda}_{\bullet}), \ 
\widetilde{\Lambda}_{\bullet} \cneq 
\lim_{\longleftarrow} 
\widetilde{\Lambda}_{\bullet}^{\le d}. 
\end{align}
The sheaf $\widetilde{\Lambda}_{\bullet}$
is a sheaf of dg-algebras on $\Spec R^{ab}$, 
which is a complex of 
left $\widetilde{R}$-modules. 
Note that 
\begin{align}\label{af:com:dg}
\Spec \Lambda_{\bullet}^{ab} \cneq 
(\Spec R^{ab}, 
(\widetilde{\Lambda}_{\bullet})^{ab})
\end{align}
is an affine 
commutative dg-scheme. 
We 
call 
the dg-ringed space (\ref{affine:ncdg}) 
as an \textit{affine NCDG scheme}
or 
an \textit{affine
NCDG structure} on
the commutative dg-scheme (\ref{af:com:dg}). 
We call it \textit{smooth} if the ungraded 
algebra $R$ is smooth. 
In this case, (\ref{af:com:dg}) is a smooth affine 
commutative dg-scheme. 
The \textit{zero-th truncation} of (\ref{affine:ncdg}) 
is defined by 
\begin{align}\notag
\tau_0(\Spf \Lambda_{\bullet}) \cneq 
(\Spec \hH_0(\Lambda_{\bullet}^{ab}), 
\lim_{\longleftarrow}\hH_0(\widetilde{\Lambda}^{\le d}_{\bullet})). 
\end{align}
Since $\hH_0(\Lambda_{\bullet}^{\le d})=\hH_0(\Lambda_{\bullet})^{\le d}$, 
we have 
$\tau_0(\Spf \Lambda_{\bullet})=\Spf \hH_0(\Lambda_{\bullet})$, 
which is an affine  
NC structure on $\Spec \hH_0(\Lambda_{\bullet}^{ab})$. 
\begin{exam}\label{exam:xy}
Let $R=\mathbb{C}[x]$ and 
set $W_{\bullet}=W_{-1}=\mathbb{C} \cdot y$. 
Then $\Lambda_{\bullet}=\mathbb{C}\langle x, y \rangle$
where $x$ is degree zero and $y$ is degree $-1$. 
For $n\ge 1$, let $Q$ be the differential given by 
\begin{align}\label{ex:Cxy}
Q \colon \Lambda_{\bullet} \to \Lambda_{\bullet+1}, \ 
Q(x)=0, \ Q(y)=x^n. 
\end{align}
We have the associated affine NCDG scheme
\begin{align}\label{ex:ncdg}
\Spf \Lambda_{\bullet}=
(\Spec \mathbb{C}[x], \widetilde{\mathbb{C}\langle x, y \rangle})
\end{align}
which is an affine NCDG structure on 
the commutative dg-scheme 
\begin{align}\label{ex:cdg}
\Spec \Lambda_{\bullet}^{ab}=
(\Spec \mathbb{C}[x], \widetilde{\mathbb{C}[x]} \oplus 
\widetilde{\mathbb{C}[x]}y).
\end{align}
The zero-th truncation of (\ref{ex:ncdg}) is 
$\Spec \mathbb{C}[x]/(x^n)$.  
\end{exam}
We define the following NCDG analogue of NC schemes. 
\begin{defi}
A dg-ringed space $(N, \oO_{N, \bullet}^{\n})$
is called an NCDG scheme if it is locally 
isomorphic to affine NCDG schemes. 
\end{defi}
By Lemma~\ref{lem:mult}, 
for any open subset $U \subset \Spec R^{ab}$, 
the restriction $(U, \widetilde{\Lambda}_{\bullet}|_{U})$
is an NCDG scheme. 
We can also define the NCDG analogue of Definition~\ref{def:QNC}: 
\begin{defi}\label{def:QNC2}
Let $(N, \oO_{N, \bullet})$ be a commutative dg-scheme. 
A quasi NCDG structure on $(N, \oO_{N, \bullet})$
consists of an affine 
open cover $\{U_i\}_{i\in \mathbb{I}}$
of $N$, 
affine NCDG structures
$(U_{i}, \oO_{U_i, \bullet}^{\rm{nc}})$
on $(U_i, \oO_{N, \bullet}|_{U_i})$
for each $i\in \mathbb{I}$, 
and isomorphisms
of NCDG schemes
\begin{align}\label{phiij}
\phi_{ij, \bullet} \colon 
(U_{ij}, \oO_{U_j, \bullet}^{\rm{nc}}|_{U_{ij}})
\stackrel{\cong}{\to}
(U_{ij}, \oO_{U_i, \bullet}^{\rm{nc}}|_{U_{ij}}). 
\end{align}
satisfying $(\phi_{ij, \bullet})^{ab}=\id$. 
 \end{defi}
A quasi NCDG structure in Definition~\ref{def:QNC2}
is called \textit{smooth} if each $(U_{i}, \oO_{U_i, \bullet}^{\rm{nc}})$
is smooth.
If $(N, \oO_{N, \bullet})$ admits a smooth quasi NCDG structure, then 
it must be smooth. 
If the isomorphisms (\ref{phiij})
satisfy the cocycle condition, 
the sheaves of dg-algebras
$\oO_{U_i, \bullet}^{\n}$ glue to give 
the sheaf of dg-algebras $\oO_{N, \bullet}^{\n}$ on $N$. 
In this case, a pair $(N, \oO_{N, \bullet}^{\n})$
is an NCDG scheme, and 
called a \textit{NCDG structure} on $(N, \oO_{N, \bullet})$. 
Let $M \subset N$ be the closed subscheme 
given by the zero-th truncation of $(N, \oO_{N, \bullet})$, 
and set $V_i=M \cap U_i$
for the quasi NCDG structure
in Definition~\ref{def:QNC2}. 
By the definition, 
we have the induced quasi 
NC structure  
\begin{align}\label{trunc}
\{V_i^{\n}=\tau_0(U_i, \oO_{U_i, \bullet}^{\n})\}_{i\in \mathbb{I}}, \ 
\hH_0(\phi_{ij, \bullet}) \colon 
V_j^{\n}|_{V_{ij}} \stackrel{\cong}{\to}
V_i^{\n}|_{V_{ij}}
\end{align}
on $M$
in the sense of Definition~\ref{def:QNC}. 
We call the quasi NC structure (\ref{trunc}) as 
the \textit{zero-th truncation} of the quasi NCDG structure 
in Definition~\ref{def:QNC2}.

\subsection{NC virtual structure sheaves}
Let $(N, \oO_{N, \bullet})$ be a smooth commutative 
dg-scheme, 
which admits  
a smooth quasi NCDG structure in Definition~\ref{def:QNC2}. 
For simplicity, suppose that 
the quasi NCDG structure glues to give an NCDG structure
$(N, \oO_{N, \bullet}^{\n})$ on $(N, \oO_{N, \bullet})$.
Then 
its zero-th truncation is an NC 
structure $M^{\n}=(M, \oO_M^{\n})$ on $M$. As an analogy of 
the identity $\oO_M^{\rm{vir}}
=[\hspace{-0.5mm}[\oO_{N, \bullet}]\hspace{-0.5mm}]$
in Theorem~\ref{thm:vir},  
one would like to define the `NC virtual structure sheaf'
of $M^{\n}$ 
to be
\begin{align}\label{ncvir:1}
\oO_{M}^{\rm{ncvir}} =
\sum_{i\in \mathbb{Z}} (-1)^i [\hH_i(\oO_{N, \bullet}^{\n})]. 
\end{align}
Note that each $\hH_i(\oO_{N, \bullet}^{\n})$
is a quasi coherent left $\oO_M^{\n}$-module. 
Hence if they are coherent and vanish for $\lvert i \rvert \gg 0$, 
then the sum (\ref{ncvir:1}) 
makes sense as an element
of $K_0(M^{\n})$, where 
$K_0(M^{\n})$ is the Grothendieck group of 
the category of coherent left $\oO_M^{\n}$-modules. 
However in general, the 
sum (\ref{ncvir:1}) is an infinite sum 
even if 
$(N, \oO_{N, \bullet})$ is a $[0, 1]$-manifold. 
For example, the cohomologies of the 
complex (\ref{ex:Cxy}) 
are not bounded while the commutative dg-scheme (\ref{ex:cdg})
is a
$[0, 1]$-manifold. 

Instead of
$(N, \oO_{N, \bullet}^{\n})$, consider 
the NCDG scheme $(N, (\oO_{N, \bullet}^{\n})^{\le d})$
for $d \in \mathbb{Z}_{\ge 0}$. It is 
regarded as a $d$-smooth dg-resolution of
the
$d$-th order NC thickening $(M^{\n})^{\le d}=(M, (\oO_M^{\n})^{\le d})$ 
of $M$. 
Moreover we have the following: 
\begin{lem}
If $(N, \oO_{N, \bullet})$ is a 
$[0, 1]$-manifold, then  
$(\oO_{N, \bullet}^{\n})^{\le d}$ is 
quasi-isomorphic to a 
bounded complex. 
\end{lem}
\begin{proof}
The subquotient
of the NC 
filtration of $(\oO_{N, \bullet}^{\n})^{\le d}$
is the direct sum of 
$\gr_F(\oO_{N, \bullet}^{\n})^j_{\bullet}$ 
for $0\le j\le d$. 
It is easy to see that $\gr_F(\oO_{N, \bullet}^{\n})^j_{\bullet}$ is a 
finitely generated dg $\oO_{N, \bullet}$-module, hence 
bounded if $(N, \oO_{N, \bullet})$ is a $[0, 1]$-manifold. 
Therefore the lemma holds. 
\end{proof} 
By the above lemma, the sum 
\begin{align}\label{ncvir:2}
(\oO_{M}^{\rm{ncvir}})^{\le d} =
\sum_{i\in \mathbb{Z}} (-1)^i [\hH_i((\oO_{N, \bullet}^{\n})^{\le d})]
\end{align}
makes sense in $K_0((M^{\n})^{\le d})$, 
which is identified as an element of $K_0(M)$
since any left $(\oO_M^{\n})^{\le d}$-module 
has a finite filtration whose subquotients
are $\oO_M$-modules. 
We call (\ref{ncvir:2}) as 
\textit{$d$-th NC virtual structure sheaf} 
of the $d$-th NC thickening $(M^{\n})^{\le d}$
of $M$. 

In general, 
a quasi NCDG structure in Definition~\ref{def:QNC2}
may not glue to 
give an NCDG structure. 
In such a case, the 
zero-th truncation (\ref{trunc})
only gives a quasi NC 
structure $M^{\n}$ on $M$. 
We generalize the 
above notion of $d$-th NC virtual structure 
sheaves to the quasi NC structure $M^{\n}$. 
Even if the isomorphisms (\ref{phiij}) 
do not satisfy the cocycle condition, 
we have the following lemma: 
\begin{lem}
The isomorphisms (\ref{phiij})
induce the isomorphisms
\begin{align*}
\gr_F(\phi_{ij, \bullet})^d_{\bullet} \colon 
\gr_F(\oO_{U_j, \bullet}^{\rm{nc}}|_{U_{ij}})^d_{\bullet}
\stackrel{\cong}{\to}
\gr_F(\oO_{U_j, \bullet}^{\rm{nc}}|_{U_{ij}})^d_{\bullet}
\end{align*}
of dg-$\oO_{N, \bullet}|_{U_{ij}}$-modules
satisfying the cocycle condition. 
\end{lem}
\begin{proof}
Let $\Lambda_{\bullet}$ be a graded algebra and 
$\phi_{\bullet} \colon \Lambda_{\bullet} \to \Lambda_{\bullet}$
an isomorphism 
of graded algebras
satisfying $\phi^{ab}_{\bullet}=\id$. 
Then the 
induced isomorphism 
$\gr_F(\phi_{\bullet}) 
\colon \gr_F(\Lambda_{\bullet}) \to \gr_F(\Lambda_{\bullet})$
is the identity 
by~\cite[Lemma~2.2]{Todnc}. 
The lemma obviously follows from this fact.  
\end{proof}
By the above lemma, 
the sheaves
$\gr_F(\oO_{U_i, \bullet}^{\n})^d_{\bullet}$ glue to 
give the global 
dg-$\oO_{N, \bullet}$-module on $N$
\begin{align}\label{gr:global}
\gr_F(\oO_{N, \bullet}^{\n})^d_{\bullet} \in 
\rm{dg} \ 
\oO_{N, \bullet} \mbox{-}\modu.
\end{align}
Since (\ref{gr:global}) is finitely generated, 
we have the following element
\begin{align*}
[\hspace{-0.5mm}[\gr_F(\oO_{N, \bullet}^{\n})^d_{\bullet}]\hspace{-0.5mm}]
 \in K_0(M). 
\end{align*}
Hence
the following definition makes sense: 
\begin{defi}\label{defi:ncvir}
Let
$(N, \oO_{N, \bullet})$
be a $[0, 1]$-manifold and $M \subset N$ its zero-th truncation. 
Suppose that it admits a quasi NCDG structure, 
and 
let $M^{\n}$ be its zero-th truncation (\ref{trunc}). 
The $d$-th NC virtual structure sheaf of $M^{\n}$ is defined
to be
\begin{align}\label{ncvir:3}
(\oO_{M}^{\rm{ncvir}})^{\le d} \cneq 
\sum_{j=0}^{d}[\hspace{-0.5mm}[\gr_F(\oO_{N, \bullet}^{\n})^j_{\bullet}]
\hspace{-0.5mm}] \in K_0(M). 
\end{align}
\end{defi}
\begin{rmk}
If a quasi NCDG structure gives 
the NCDG structure, then the class (\ref{ncvir:3})
coincides with (\ref{ncvir:2}) by taking the NC filtration of 
$(\oO_{N, \bullet}^{\n})^{\le d}$. 
\end{rmk}

\begin{rmk}
By Theorem~\ref{thm:vir}, 
for $d=0$ we have the identity 
$(\oO_{M}^{\rm{ncvir}})^{\le 0}=\oO_M^{\rm{vir}}$, 
where $\oO_M^{\rm{vir}}$ is the commutative 
virtual structure sheaf given in (\ref{Kth:vir}). 
\end{rmk}

\begin{exam}
In the situation of Example~\ref{exam:xy}, 
let 
$M=\Spec \mathbb{C}[x]/(x^n)$. 
By definition, the
$d$-th NC virtual structure sheaf of 
$M$ is 
\begin{align}\label{ex:ncvir}
(\oO_M^{\rm{ncvir}})^{\le d}=\sum_{j=0}^{d}
\left[ \hspace{-0.5mm} \left[
\gr_F \left( \mathbb{C}\langle x, y \rangle \right)^j \right] \hspace{-0.5mm} 
\right]. 
\end{align}
By the identification 
$K_0(M)=\mathbb{Z}$, we have 
$\oO_M^{\rm{vir}}=[\hspace{-0.5mm}[\mathbb{C}[x, y]]\hspace{-0.5mm}]=n$ and 
\begin{align*}
(\oO_M^{\rm{ncvir}})^{\le 1}&
=[\hspace{-0.5mm}[\mathbb{C}[x, y]\oplus \mathbb{C}[x, y][x, y]
\oplus \mathbb{C}[x, y][y, y]]\hspace{-0.5mm}] \\
&=n. 
\end{align*}
In general, one can show that 
(\ref{ex:ncvir}) coincides with $n$ for all $d\ge 1$
(cf.~Corollary~\ref{cor:vir0}). 
\end{exam}

\section{Description of NC virtual structure sheaves}\label{sec:Des}
In this section, we give  
an explicit description of the NC virtual structure sheaves
in terms of the perfect obstruction theory (\ref{perfect}), 
and prove Theorem~\ref{thm:intro:perf}. 
\subsection{Graded Poisson envelope}
Let $\Lambda_{\bullet}$ be a graded algebra (\ref{Lambda:bull}), 
and take the NC filtration (\ref{NCfilt}).
For $x \in (F^d \Lambda_{\bullet})_{e}$, 
$x' \in (F^{d'} \Lambda_{\bullet})_{e'}$, 
it is easy to see that
\begin{align*}
x \cdot x' \in 
(F^{d+d'} \Lambda_{\bullet})_{e+e'}, \ 
[x, x'] \in (F^{d+d'+1} \Lambda_{\bullet})_{e+e'}. 
\end{align*} 
Here $e, e'$ are $|_{\bullet}$-gradings. 
Therefore the bracket $[-, -]$ induces the pairing
\begin{align*}
\{-, -\} \colon 
\gr_F(\Lambda_{\bullet})_{e}^d 
\times \gr_F(\Lambda_{\bullet})_{e'}^{d'} \to 
\gr_F(\Lambda_{\bullet})_{e+e'}^{d+d'+1}
\end{align*}
which is super anti-symmetric
with respect to the 
$\lvert_{\bullet}$-grading. 
In general, 
we introduce the following definition: 
\begin{defi}
(i) 
A $\lvert_{\bullet}$-graded Poisson algebra 
is a 
triple
\begin{align}\label{triple}
(P_{\bullet}, \cdot, \{-, -\})
\end{align}
where $(P_{\bullet}, \cdot)$
is a super commutative graded algebra, 
$\{-, -\}$ is a grade preserving, 
super anti-symmetric 
 pairing
\begin{align}\label{pairing}
\{-, -\} \colon P_{\bullet} \times P_{\bullet} \to 
P_{\bullet}
\end{align}
satisfying the 
super Jacobi identity 
and $\{x, -\}$ is a super derivation for any $x \in P_{\bullet}$. 

(ii) A $|_{\bullet}^{\bullet}$-graded Poisson algebra is a 
$\lvert_{\bullet}$-graded Poisson algebra 
$(P_{\bullet}, \cdot, \{-, -\})$
endowed with another grading 
(called $\lvert^{\bullet}$-grading)
$P_{\bullet}^{\bullet}$
such that
the multiplication $\cdot$ 
preserves the $\lvert^{\bullet}$-degree, and 
 the paring (\ref{pairing})
sends $P_{e}^{d} \times P_{e'}^{d'}$ to 
$P_{e+e'}^{d+d'+1}$. 
\end{defi}
A $\lvert_{\bullet}$-graded Poisson algebra
$P_{\bullet}$
satisfying 
$P_i=0$ for $i\neq 0$
is nothing but a 
\textit{Poisson algebra}. 
It is easy to see that the algebra (\ref{grFlam})
is a $|_{\bullet}^{\bullet}$-graded Poisson algebra. 
Also a $|_{\bullet}^{\bullet}$-graded 
Poisson algebra $(P_{\bullet}^{\bullet}, \cdot, \{-, -\})$
is interpreted as a
$|_{\bullet}$-graded Poisson algebra by forgetting 
the $\lvert^{\bullet}$-degree. 
We consider the functor
\begin{align}\label{Pois:fun}
(\rvert_{\bullet}\mbox{-graded Poisson algebras})
\to (\mbox{super commutative graded algebras})
\end{align}
defined by 
forgetting (\ref{pairing}), i.e. it 
sends 
the triple (\ref{triple})
to the super commutative graded algebra
$(P_{\bullet}, \cdot)$. 
The functor (\ref{Pois:fun}) 
extends the
forgetting functor from the 
category of Poisson algebras to
the category of commutative algebras. 
It is well-known that the latter functor 
has a left adjoint, called \textit{Poisson envelope}. 
We see that the graded version of the 
similar construction 
gives the left adjoint of (\ref{Pois:fun}). 

For a graded vector space $W_{\bullet}$,
let 
\begin{align}\label{def:LW}
L(W_{\bullet}) \subset 
T(W_{\bullet})
\end{align} 
be the $|_{\bullet}$-graded 
super Lie algebra generated by $W_{\bullet}$. 
It is a direct sum of $L_{e}^{d}(W_{\bullet})$, where 
$e$ is the $|_{\bullet}$-grading, and 
$d$ is the grading 
(called 
$|^{\bullet}$-grading)
 determined by 
\begin{align}\label{upper:grade}
L^0_{\bullet}(W_{\bullet})=W_{\bullet}, \ 
L^{d+1}_{\bullet}(W_{\bullet})=[W_{\bullet}, L^d_{\bullet}(W_{\bullet})].
\end{align}
For example, $L^1_{e}(W_{\bullet})$ 
is spanned by
\begin{align*}
[x_0, x_1]=x_0 \otimes x_1-(-1)^{e_0 e_1} x_1 \otimes x_0, \ 
x_i \in W_{e_i}, \ e_0+e_1=e.
\end{align*}
\begin{defi}
The \textit{free Poisson algebra} generated by $W_{\bullet}$ is defined by
\begin{align}\label{free:Poiss}
\mathrm{Poiss}(W_{\bullet}) \cneq SL(W_{\bullet})
\end{align}
where $S(-)$ is the super symmetric product with 
respect to the $|_{\bullet}$-grading. 
\end{defi}
Note that (\ref{free:Poiss}) 
is tri-graded:
it is the 
direct sum of $S^n L(W_{\bullet})_{e}^{d}$
spanned by 
elements of the form
\begin{align*}
\prod_{i=1}^{n}
[x_0^{(i)}, [x_1^{(i)}, \cdots, [x_{d_i-1}^{(i)}, x_{d_i}^{(i)}]\cdots ]]
\end{align*}
for $x_{j}^{(i)} \in W_{e_{ij}}$
with $1\le i\le n$, $0\le j\le d_i$ satisfying
\begin{align*}
\sum_{i=1}^{n}d_i=d, \quad
 \sum_{i=1}^{n} \sum_{j=1}^{d_i} e_{ij}=e. 
\end{align*}
Here $e$ is $|_{\bullet}$-degree, 
$d$ is $|^{\bullet}$-degree, and 
$n$ is called $^{\bullet}|$-degree. 
By the $|_{\bullet}$-graded Leibniz rule, 
the bracket $[-, -]$
on $L(W_{\bullet})$ extends 
to the $|_{\bullet}$-graded bracket $\{-, -\}$
on 
(\ref{free:Poiss}). 
Then the algebra (\ref{free:Poiss})
is a $|_{\bullet}^{\bullet}$-graded Poisson algebra 
with respect to the gradings $|^{\bullet}$
and $|_{\bullet}$. 

Let $A_{\bullet}$ be a
super commutative graded algebra. 
By regarding it as a graded vector space, 
we obtain the $|_{\bullet}^{\bullet}$-graded Poisson algebra 
$\mathrm{Poiss}(A_{\bullet})$. 
Let $I_{A_{\bullet}} \subset S(A_{\bullet})$ be the 
ideal given by the exact sequence
\begin{align*}
0 \to I_{A_{\bullet}} \to S(A_{\bullet}) \stackrel{\eta}{\to}
 A_{\bullet} \to 0. 
\end{align*}
Here $\eta$ is given by 
the multiplication in $A_{\bullet}$. 
The ideal $I_{A_{\bullet}}$ is generated by 
elements of the form $a \cdot b -ab$ for 
$a, b \in A_{\bullet}$. 
We then define the ideal 
\begin{align}\label{ideal:IA}
\langle \hspace{-0.5mm} \langle 
I_{A_{\bullet}}
\rangle \hspace{-0.5mm}
\rangle \subset \mathrm{Poiss}(A_{\bullet})
\end{align}
to be generated by elements of the form
\begin{align*}
\{x_1, \{x_2, \cdots, \{x_{k-1}, y\} \cdots \} \}, \ 
x_1, \cdots, x_{k-1} \in A_{\bullet}, \ y \in I_{A_{\bullet}}. 
\end{align*}
\begin{defi}
The $|_{\bullet}^{\bullet}$-\textit{graded
Poisson envelope} of $A_{\bullet}$ is defined by
\begin{align}\label{bi-Poi}
P(A_{\bullet})
\cneq \mathrm{Poiss}(A_{\bullet})/\langle \hspace{-0.5mm}
 \langle 
I_{A_{\bullet}}
\rangle \hspace{-0.5mm} \rangle. 
\end{align}
\end{defi}
The ideal (\ref{ideal:IA})
is homogeneous with respect to the
$|_{\bullet}$ and $|^{\bullet}$ grading
(but not for $^{\bullet}|$-grading), 
so the algebra (\ref{bi-Poi})
is $|_{\bullet}^{\bullet}$-graded:
it is the direct sum of $P(A_{\bullet})_{e}^d$, where $d$ is
$|^{\bullet}$-degree
and $e$ is $|_{\bullet}$-degree. 
By the definition of (\ref{ideal:IA}), 
the Poisson structure on $\mathrm{Poiss}(A_{\bullet})$
descends to the $|_{\bullet}^{\bullet}$-graded Poisson structure on 
(\ref{bi-Poi}). 
By forgetting the $|^{\bullet}$-grading on 
(\ref{bi-Poi}), we 
obtain the $|_{\bullet}$-graded 
Poisson algebra $P(A_{\bullet})_{\bullet}$. 
\begin{lem}
The functor 
$A_{\bullet} \mapsto P(A_{\bullet})_{\bullet}$
is the left adjoint of (\ref{Pois:fun}). 
\end{lem}
\begin{proof}
The proof is straightforward and left to the reader. 
\end{proof}
\begin{exam}
Let $W_{\bullet}$ be a finite dimensional graded vector space
and 
set
\begin{align*}
A_{\bullet}=S(W_{\bullet}). 
\end{align*}
Note that $A_{\bullet}$ is a super commutative 
graded algebra. 
In this case, we have 
the canonical 
isomorphism of
$|_{\bullet}$-graded
Poisson algebras
\begin{align}\label{isom:PG}
P(A_{\bullet})_{\bullet} \stackrel{\cong}{\to} 
\gr_F
T(W_{\bullet})_{\bullet}.  
\end{align}
The above isomorphism is proved 
in~\cite[Example~4.1.2]{Kap17}
when $W_{\bullet}$ consists of degree zero part, 
and the same argument works in the general case.  
\end{exam}

\subsection{Description of graded Poisson envelopes}
Let $A_{\bullet}$ be a super commutative graded algebra 
given by
\begin{align}\label{Lam:ab}
A_{\bullet} =R \otimes S(W_{\bullet})
\end{align}
for 
a smooth commutative algebra $R$ and 
a finite dimensional graded vector space $W_{\bullet}$. 
Let $\Omega_{A_{\bullet}}$ be the graded module of 
differential forms on $A_{\bullet}$. 
Similarly to (\ref{def:LW}), let 
\begin{align}\label{def:LA}
L_{A_{\bullet}}(\Omega_{A_{\bullet}}) \subset 
\bigoplus_{n\ge 0} 
\overset{n}{\overbrace{\Omega_{A_{\bullet}} \otimes_{A_{\bullet}}
 \cdots \otimes_{A_{\bullet}}
\Omega_{A_{\bullet}}}}
\end{align}
be the super $A_{\bullet}$-Lie subalgebra generated by 
$\Omega_{A_{\bullet}}$. 
It has $|_{\bullet}$-grading induced by the grading on $W_{\bullet}$, 
and also $|^{\bullet}$-grading 
similarly to (\ref{upper:grade}). 
Let $L_{A_{\bullet}}^{+}(\Omega_{A_{\bullet}})$ 
be the positive degree part 
of (\ref{def:LA})
with respect to the $|^{\bullet}$-grading. 
\begin{lem}\label{isom:bigra}
If $R$ is local, we have an isomorphism of $|_{\bullet}^{\bullet}$-graded 
Poisson algebras
\begin{align*}
P(A_{\bullet})_{\bullet}^{\bullet} \cong
S_{A_{\bullet}}
(L_{A_{\bullet}}^{+}(\Omega_{A_{\bullet}}))_{\bullet}^{\bullet}. 
\end{align*}
\end{lem}
\begin{proof}
The case of $W_{\bullet}=0$ is 
proved in~\cite[Theorem~1.4]{GC}. 
The case of $W_{\bullet} \neq 0$ is similarly
proved without any modification. Indeed
let ${\bf m} \subset R$ be
the maximal ideal of $R$, 
and set $V={\bf m}/{\bf m}^2$. 
We have the identification
\begin{align}\label{id:omega}
\Omega_{A_{\bullet}}=A_{\bullet} \otimes (V \oplus W_{\bullet}).
\end{align}
By the identification (\ref{id:omega}), 
we have 
\begin{align*}
S_{A_{\bullet}}(L_{A_{\bullet}}^{+}(\Omega_{A_{\bullet}}))
=A_{\bullet} \otimes S L^{+}(V \oplus W_{\bullet}). 
\end{align*}
The above identification 
gives a substitute of the third line 
of the proof of~\cite[Theorem~1.4]{GC}, and 
the rest of the proof is the same. 
\end{proof}
Next we consider the case that $R$ is not necessary local. 
Let
\begin{align}\label{pi}
\pi \colon S L (A_{\bullet})_{\bullet}^{\bullet} \to 
P(A_{\bullet})_{\bullet}^{\bullet}
\end{align}
be the natural projection. The map (\ref{pi}) 
vanishes on $I_A$, hence it 
factors through the map
\begin{align*}
\overline{\pi} \colon A_{\bullet} \otimes 
S L^{+}(A_{\bullet})_{\bullet}^{\bullet}
\twoheadrightarrow P(A_{\bullet})_{\bullet}^{\bullet}. 
\end{align*}
The above map 
preserves
both of $|_{\bullet}$ and $|^{\bullet}$ degrees.  
For $n\ge 0$, let $I^n$ be the ideal of $P(A_{\bullet})_{\bullet}^{\bullet}$
\begin{align}\label{ideal:I}
I^n \cneq \overline{\pi} \left( A_{\bullet} \otimes \bigoplus_{p\ge n}
S^p L^{+}(A_{\bullet})_{\bullet}^{\bullet}  \right). 
\end{align}
The ideal $I^n$ is homogeneous in both of 
$|_{\bullet}$ and $|^{\bullet}$ degrees. 
We define
\begin{align}\label{def:G}
GP(A_{\bullet})^{\bullet}_{\bullet}
\cneq \bigoplus_{n\ge 0}
I^n/I^{n+1}. 
\end{align}
Note that (\ref{def:G}) is a tri-graded algebra:
it is a direct sum of 
\begin{align}\label{Ged}
G^{n} P(A_{\bullet})_{e}^d \cneq
\left( I^n/I^{n+1} \right)_{e}^{d}
\end{align}
where $e$ is $|_{\bullet}$-degree and 
$d$ is $|^{\bullet}$-degree. 
The degree $n$ is called $^{\bullet}|$-degree. 
Note that (\ref{Ged}) vanishes for $n>d$ so 
the filtration
\begin{align}\label{stabilize}
P(A_{\bullet})_{\bullet}^{d} = (I^0)^{d}_{\bullet} \supset
(I^1)^d_{\bullet} \supset
\cdots \supset (I^n)^d_{\bullet} \supset \cdots
\end{align}
stabilizes for $n>d$. 
\begin{lem}\label{lem:tri:gra}
There is a natural isomorphism of 
tri-graded algebras
\begin{align}\label{tri:gra}
S_{A_{\bullet}}^{\bullet}L_{A_{\bullet}}^{+}
(\Omega_{A_{\bullet}})_{\bullet}^{\bullet}
\stackrel{\cong}{\to}
G^{\bullet}P(A_{\bullet})^{\bullet}_{\bullet}. 
\end{align}
\end{lem}
\begin{proof}
If $W_{\bullet}$ consists of degree zero part, then 
$|_{\bullet}$-degrees of both 
sides of (\ref{tri:gra}) 
consists of degree zero, and 
the result is proved in~\cite[Theorem~1.6]{GC}. 
The case of general $|_{\bullet}$-graded $W_{\bullet}$ is 
similarly proved without any modification. 
\end{proof}

\subsection{Graded NC filtration via graded Poisson envelope}\label{subsec:via}
Let $R$ be a smooth
(not necessary commutative) algebra, and
$W_{\bullet}$ a finite dimensional graded vector space. 
We 
 consider
 the graded algebra 
$\Lambda_{\bullet}$ given by
(\ref{Lambda:bull}), and set
\begin{align*}
A_{\bullet} \cneq (\Lambda_{\bullet})^{ab}. 
\end{align*}
By (\ref{abeliz}), 
the algebra $A_{\bullet}$ is a super commutative graded algebra
of the form (\ref{Lam:ab}).  
We have the following lemma: 
\begin{lem}\label{lem:isom:Poi}
We have the canonical 
isomorphism of $|_{\bullet}^{\bullet}$-graded Poisson algebras
\begin{align*}
P(A_{\bullet})^{\bullet}_{\bullet} \stackrel{\cong}{\to} 
\gr_F(\Lambda_{\bullet})_{\bullet}^{\bullet}. 
\end{align*}
\end{lem}
\begin{proof}
If $W_{\bullet}=0$, the result is 
the consequence of~\cite[Theorem~4.2.1]{Kap17}, 
and almost the same proof is applied
for $W_{\bullet} \neq 0$. 
Note that $\gr_F(\Lambda_{\bullet})^0=A_{\bullet}$. 
By the universality of the $|_{\bullet}$-graded Poisson envelope, we 
have the canonical morphism of $|_{\bullet}$-graded
Poisson algebras
$\phi \colon P(A_{\bullet})_{\bullet} \to 
\gr_F(\Lambda_{\bullet})_{\bullet}$, 
which also preserves $|^{\bullet}$-grading.  
Since both of
$P(A_{\bullet})^{\bullet}_{\bullet}$ and 
$\gr_F(\Lambda_{\bullet})^{\bullet}_{\bullet}$
are $|_{\bullet}$-graded $A_{\bullet}$-modules, 
we can interpret them as $R^{ab}$-modules 
by the algebra
homomorphism 
\begin{align*}
R^{ab} \to A_{\bullet}, \ 
\lambda \mapsto \lambda \otimes 1.
\end{align*}
 It is enough to show that, 
for any closed point $x \in \Spec R^{ab}$, 
$\phi$ induces
the isomorphism
\begin{align}\label{induced}
P(A_{\bullet})^{\bullet}_{\bullet} \otimes_{R^{ab}}
R^{ab}/{\bf m}_x \stackrel{\cong}{\to} 
\gr_F(\Lambda_{\bullet})^{\bullet}_{\bullet}
\otimes_{R^{ab}}
R^{ab}/{\bf m}_x.
\end{align}
Here ${\bf m}_x \subset R^{ab}$ is the maximal ideal 
which defines $x$. 
Let $A_{x, \bullet} \cneq 
\Lambda_{x}^{ab} \otimes S(W_{\bullet})$. 
Then we have
\begin{align*}
P(A_{\bullet})^{\bullet}_{\bullet} \otimes_{R^{ab}}
R^{ab}/{\bf m}_x=
P(A_{x, \bullet})^{\bullet}_{\bullet} \otimes_{R^{ab}}
R^{ab}/{\bf m}_x.
\end{align*}
By Lemma~\ref{isom:bigra}, the RHS is computed as
\begin{align}\label{isom:P}
P(A_{x, \bullet})^{\bullet}_{\bullet} \otimes_{R^{ab}}
R^{ab}/{\bf m}_x=
S(W_{\bullet})_{\bullet} \otimes 
S L^{+}(V \oplus W_{\bullet})_{\bullet}^{\bullet}.
\end{align}
Here $V={\bf m}_x/{\bf m}_x^2$. 
Applying the same argument for 
$S(V) \otimes S(W_{\bullet})=S(V\oplus W_{\bullet})$, 
we obtain
\begin{align}\label{isom:P2}
P(S(V \oplus W_{\bullet}))^{\bullet}_{\bullet}
\otimes_{S(V)} S(V)/{\bf m}_0
=S(W_{\bullet})_{\bullet} \otimes 
S L^{+}(V \oplus W_{\bullet})_{\bullet}^{\bullet}.
\end{align}
Here ${\bf m}_0$ is the maximal ideal of $S(V)$ 
corresponding to the origin. 
On the other hand, by~\cite[Lemma~4.2.2]{Kap17}, we have 
the isomorphism for $j>d$
\begin{align*}
\gr_F(\Lambda_{\bullet})^{d}_{\bullet}
\otimes_{R^{ab}}
R^{ab}/{\bf m}_x
\stackrel{\cong}{\to}
\gr_F(\Lambda_{\bullet}/{\bf m}_{x}^j)^{d}_{\bullet}
\otimes_{R^{ab}}
R^{ab}/{\bf m}_x.
\end{align*}
Since we have
\begin{align*}
\Lambda_{\bullet}/{\bf m}_x^j \cong
T(V \oplus W_{\bullet})/{\bf m}_0^j
\end{align*}
we have the identification
\begin{align}\label{isom:g}
\gr_F(\Lambda_{\bullet})_{\bullet}^{\bullet}
\otimes_{R^{ab}}
R^{ab}/{\bf m}_x =
\gr_F(T(V \oplus W_{\bullet}))^{\bullet}_{\bullet}
\otimes_{S(V)}
S(V)/{\bf m}_0. 
\end{align}
By (\ref{isom:P}), (\ref{isom:P2}), 
(\ref{isom:g}) and 
the isomorphism (\ref{isom:PG}), 
$\phi$ induces the isomorphism (\ref{induced}). 
\end{proof}

\subsection{NC virtual structure sheaves via perfect obstruction theory}
\label{subsec:viaperf}
We now return to the situation of Definition~\ref{defi:ncvir}.
Similarly to (\ref{def:LW}), (\ref{def:LA}), for a 
graded vector bundle $\pP_{\bullet} \to M$ on a scheme $M$, 
we set
\begin{align}\label{LOE}
L_{\oO_M}(\pP_{\bullet}) \subset T_{\oO_M}(\pP_{\bullet})
\end{align}
to be 
the sheaf of super $\oO_M$-Lie 
algebras generated by $\pP_{\bullet}$, i.e. 
each fiber of $L_{\oO_M}(\pP_{\bullet})$ at $x \in M$
is the super 
Lie algebra $L(\pP_{\bullet}|_{x})$.
Note that the grading on $\pP_{\bullet}$ induces the $|_{\bullet}$-grading on 
(\ref{LOE}). 
Similarly to (\ref{upper:grade}), we also have the $|^{\bullet}$-grading on
 (\ref{LOE}), 
and denote by $L_{\oO_M}^{+}(\pP_{\bullet})$ its positive degree part
with respect to the $|^{\bullet}$-grading. 
The following is the main result in this section: 
\begin{thm}\label{thm:formula}
In the situation of Definition~\ref{defi:ncvir}, we have the 
following formula in $K_0(M)$: 
\begin{align}\label{ncvir:formula}
(\oO_{M}^{\rm{ncvir}})^{\le d}=\oO_M^{\rm{vir}} \otimes_{\oO_M}
[S_{\oO_M}L_{\oO_M}^{+}(\eE_{\bullet})^{\le d}_{\bullet}]. 
\end{align}
Here
$\eE_{\bullet}$
is the two term complex (\ref{E:qis}), and 
 $(-)^{\le d}$ is the degree $\le d$ part 
with respect to the $|^{\bullet}$-grading. 
\end{thm}
\begin{proof}
For a commutative dg-scheme $(N, \oO_{N, \bullet})$, 
the construction of the graded Poisson envelope
yields the sheaf of graded $\oO_{N, \bullet}$-module 
$P(\oO_{N, \bullet})^{d}_{\bullet}$
for each $d \in \mathbb{Z}_{\ge 0}$.  
By Lemma~\ref{lem:isom:Poi}, we have the 
isomorphism of
graded $\oO_{N, \bullet}$-modules
\begin{align}\label{pf:1}
P(\oO_{N, \bullet})^{\le d}_{\bullet}
\stackrel{\cong}{\to}
\gr_F(\oO_{N, \bullet}^{\n})_{\bullet}^{\le d}. 
\end{align}
Also the construction of the ideals (\ref{ideal:I})
yields 
the filtration 
of graded $\oO_{N, \bullet}$-modules 
\begin{align}\notag
P(\oO_{N, \bullet})^{\le d}_{\bullet}
=(\iI^0)_{\bullet}^{\le d} \supset (\iI^1)_{\bullet}^{\le d}
\supset \cdots \supset (\iI^n)^{\le d}_{\bullet} \supset \cdots
\end{align}
which stabilizes due to
the stabilization of (\ref{stabilize}). 
Hence using the notation (\ref{K:grade}) and (\ref{def:G}),  
we have the identity
in $K_0(\oO_{N, \bullet})$
\begin{align}\label{pf:2}
[P(\oO_{N, \bullet})^{\le d}_{\bullet}]
=[G^{\bullet}P(\oO_{N, \bullet})^{\le d}_{\bullet}]. 
\end{align} 
Then by Lemma~\ref{lem:tri:gra}, 
we have the isomorphism 
of graded $\oO_{N, \bullet}$-modules
\begin{align}\label{pf:3}
S_{\oO_{N, \bullet}}^{\bullet}L_{\oO_{N, \bullet}}^{+}
(\Omega_{N, \bullet})_{\bullet}^{\le d} \stackrel{\cong}{\to}
G^{\bullet}P(\oO_{N, \bullet})^{\le d}_{\bullet}. 
\end{align}
By (\ref{pf:1}), (\ref{pf:2}), (\ref{pf:3})
and Corollary~\ref{cor:K0}, 
we obtain the identity in $K_0(M)$:
\begin{align}\notag
(\oO_M^{\rm{ncvir}})^{\le d}=\oO_M^{\rm{vir}}
\otimes_{\oO_M} [S_{\oO_M} L_{\oO_M}^{+}
(\overline{\Omega}_{N, \bullet}|_{M})^{\le d}_{\bullet}]. 
\end{align}
We are left to show the identity
\begin{align}\label{id:S}
[S_{\oO_M} L_{\oO_M}^{+}
(\overline{\Omega}_{N, \bullet}|_{M})^{\le d}_{\bullet}]
=[S_{\oO_M} L_{\oO_M}^{+}
(\eE_{\bullet})^{\le d}_{\bullet}]. 
\end{align}
For a partition of $n$
\begin{align}\label{parti}
\lambda=(\lambda_1, \lambda_2, \cdots, \lambda_k), \ 
\lambda_1 \ge \lambda_2 \ge \cdots \ge \lambda_k
\end{align} let 
$V_{\lambda}$ be the corresponding 
irreducible representation of $S_n$. 
Let ${\bf S}_{\lambda}$ be
the Schur functor defined on the category of 
complexes of vector bundles on $M$ to itself 
(cf.~\cite[Section~2]{PrWJ}):
\begin{align}\label{Schur}
{\bf S}_{\lambda} \colon 
\pP_{\bullet} \mapsto 
(V_{\lambda} \otimes_{\oO_M} \pP_{\bullet}^{\otimes n})^{S_n}. 
\end{align}
Since the functor 
$\pP_{\bullet} \mapsto
S_{\oO_M} L_{\oO_M}^{+}(\pP_{\bullet})_{\bullet}^{\le d}$
is a polynomial functor
on the category of graded vector bundles on $M$ 
 in the sense of~\cite[Appendix~A]{IMac},  
it is described as 
\begin{align}\label{polynomial}
\bigoplus_{n\ge 0}
\bigoplus_{\lvert \lambda \rvert=n} \uU_{\bullet, \lambda}
\otimes_{\oO_M} {\bf S}_{\lambda}(\pP_{\bullet})
\end{align}
for graded vector bundles 
$\uU_{\bullet, \lambda}$
on 
$M$. 
Here for a partition (\ref{parti}), 
we set 
$\lvert \lambda \rvert=\lambda_1+ \lambda_2+ \cdots +\lambda_k$. 
Therefore the identity (\ref{id:S})
follows from 
Lemma~\ref{lem:Kth}
below. 
\end{proof}
We have used the following lemma, which is 
probably well-known, but 
include the proof because of the lack of a reference. 
\begin{lem}\label{lem:Kth}
Let $\pP_{\bullet}, \qQ_{\bullet}$ be bounded 
complexes of vector bundles on a scheme $M$ 
and $s \colon \pP_{\bullet} \to \qQ_{\bullet}$ a 
quasi-isomorphism. 
Then we have the identity 
$[{\bf S}_{\lambda}(\pP_{\bullet})]=[{\bf S}_{\lambda}(\qQ_{\bullet})]$
in $K_0(M)$. 
\end{lem}
\begin{proof}
By the representation theory of $S_n$ (cf.~\cite[Theorem~6.3]{FuHar}), 
we have the decomposition
of complexes
\begin{align*}
(\pP_{\bullet})^{\otimes n} =
\bigoplus_{\lvert \lambda \rvert=n} 
{\bf S}_{\lambda}(\pP_{\bullet}) \otimes V_{\lambda}. 
\end{align*}
Therefore 
if $\pP_{\bullet}$ is acyclic, 
then ${\bf S}_{\lambda}(\pP_{\bullet})$ is
also acyclic.  
Let $\Cone(s)_{\bullet}$ be the cone of $s$, which is 
acyclic as $s$ is 
quasi-isomorphism. 
By the above argument, the complex
${\bf S}_{\lambda}(\Cone(s)_{\bullet})$ is 
also acyclic, hence it is zero in $K_0(M)$. 
On the other hand, as 
$\Cone(s)_{\bullet}$ is 
$\qQ_{\bullet} \oplus \pP_{\bullet}[1]$ as a
graded vector bundle, 
${\bf S}_{\lambda}(\Cone(s)_{\bullet})$ is isomorphic to 
${\bf S}_{\lambda}(\qQ_{\bullet} \oplus \pP_{\bullet}[1])$
as graded vector bundles. 
We have the decomposition
as graded vector bundles (cf.~\cite[Exercise~6.11]{FuHar})
\begin{align}\label{rule:1}
{\bf S}_{\nu}(\qQ_{\bullet} \oplus \pP_{\bullet}[1])
=\bigoplus_{\lvert \lambda \rvert + \lvert \mu \rvert=\lvert \nu \rvert}
\left({\bf S}_{\lambda}(\qQ_{\bullet}) 
\otimes_{\oO_M} {\bf S}_{\mu}
(\pP_{\bullet}[1])\right)^{\oplus N_{\lambda, \mu}^{\nu}}. 
\end{align}
Here $N_{\lambda, \mu}^{\nu} \in \mathbb{Z}_{\ge 0}$
is determined by the Littlewood-Richardson rule. 
Hence we obtain the identity in $K_0(M)$:
\begin{align}\label{rule:2}
\sum_{\lvert \lambda \rvert + \lvert \mu \rvert=\lvert \nu \rvert}
N_{\lambda, \mu}^{\nu}[{\bf S}_{\lambda}(\qQ_{\bullet})] \otimes_{\oO_M}
[{\bf S}_{\mu}(\pP_{\bullet}[1])]=0. 
\end{align}
Applying the above identity to $\id \colon \pP_{\bullet} \to \pP_{\bullet}$, we obtain
\begin{align}\label{rule:3}
\sum_{\lvert \lambda \rvert + \lvert \mu \rvert=\lvert \nu \rvert}
N_{\lambda, \mu}^{\nu}[{\bf S}_{\lambda}(\pP_{\bullet})] \otimes_{\oO_M}
[{\bf S}_{\mu}(\pP_{\bullet}[1])]=0. 
\end{align}
Noting that $N_{\lambda, \emptyset}^{\lambda}=1$, 
the identities (\ref{rule:2}), (\ref{rule:3}) together with
the induction on $\lvert \lambda \rvert$
shows that 
$[{\bf S}_{\lambda}(\pP_{\bullet})]=[{\bf S}_{\lambda}(\qQ_{\bullet})]$. 
\end{proof}
Using Theorem~\ref{thm:formula}, we 
can compute NC virtual structure sheaves 
in terms of Schur complexes
${\bf S}_{\lambda}(\eE_{\bullet})$, 
given by (\ref{Schur}). 
Note that we have
\begin{align*}
{\bf S}_{\lambda}
(\eE_{\bullet})=\bigwedge^d \eE_{\bullet}, \ 
\lambda=\overset{d}{\overbrace{(1, 1, \cdots, 1)}}. 
\end{align*}
\begin{cor}\label{cor:virform}
(i) 
For $d=1$, we have the following formula: 
\begin{align*}
(\oO_M^{\rm{ncvir}})^{\le 1}=
\oO_M^{\rm{vir}} \otimes_{\oO_M}
 \left(1+ {\bf S}_{(1, 1)}(\eE_{\bullet}) \right). 
\end{align*}

(ii) For $d=2$, we have the following formula: 
\begin{align*}
(\oO_M^{\rm{ncvir}})^{\le 2}=
\oO_M^{\rm{vir}} \otimes_{\oO_M} \left(1+
{\bf S}_{(1, 1)}(\eE_{\bullet}) + {\bf S}_{(2, 1)}(\eE_{\bullet})
+{\bf S}_{(2, 2)}(\eE_{\bullet}) + {\bf S}_{(1, 1, 1, 1)}(\eE_{\bullet})
 \right). 
\end{align*}
\end{cor}
\begin{proof}
The formula (\ref{ncvir:formula})
implies
\begin{align*}
&(\oO_M^{\rm{ncvir}})^{\le 1}=
\oO_M^{\rm{vir}} \otimes_{\oO_M} \left(1+
 L_{\oO_M}(\eE_{\bullet})^1 \right) \\
&(\oO_M^{\rm{ncvir}})^{\le 2}=
\oO_M^{\rm{vir}} \otimes_{\oO_M} \left(1+
 L_{\oO_M}(\eE_{\bullet})^1+L_{\oO_M}(\eE_{\bullet})^2
+S^2_{\oO_M}L_{\oO_M}(\eE_{\bullet})^1 \right).
\end{align*}
The formula for $d=1$ then follows from 
\begin{align}\label{comp:L1}
L_{\oO_M}(\eE_{\bullet})^{1}=[\eE_{\bullet}, \eE_{\bullet}]=\bigwedge^2 \eE_{\bullet}={\bf S}_{(1, 1)}(\eE_{\bullet}). 
\end{align}
For $d=2$, 
we have the exact sequences of complexes
\begin{align*}
0 \to \bigwedge^3 \eE_{\bullet} \to \left(\bigwedge^2 \eE_{\bullet} \right)
 \otimes 
\eE_{\bullet} \to L_{\oO_M}(\eE_{\bullet})^2 \to 0 \\
0 \to {\bf S}_{(2, 1)}(\eE_{\bullet}) \to 
\left(\bigwedge^2 \eE_{\bullet} \right)\otimes 
\eE_{\bullet} \to \bigwedge^3 \eE_{\bullet} \to 0
\end{align*}
showing that 
$[L_{\oO_M}(\eE_{\bullet})^2]=[{\bf S}_{(2, 1)}(\eE_{\bullet})]$. 
Here the former sequence easily follows from 
$L_{\oO_M}(\eE_{\bullet})^2=[\eE_{\bullet}, [\eE_{\bullet}, \eE_{\bullet}]]$
and the latter sequence follows from~\cite[Section~6.1]{FuHar}. 
Also by~\cite[Exercise~6.16]{FuHar},
we have the identity 
\begin{align*}
[S^2_{\oO_M} {\bf S}_{(1, 1)}(\eE_{\bullet})]
=[{\bf S}_{(2, 2)}(\eE_{\bullet})]+[{\bf S}_{(1, 1, 1, 1)}(\eE_{\bullet})]. 
\end{align*}
By combining these identities, we obtain the desired formula for $d=2$. 
\end{proof}

\begin{cor}\label{cor:vir0}
Suppose that 
$M^{\rm{vir}}$ has virtual dimension zero. 
Then we have the identity
\begin{align}\label{id:vir0}
(\oO_{M}^{\rm{ncvir}})^{\le d}=\oO_M^{\rm{vir}}. 
\end{align}
\end{cor}
\begin{proof}
The assumption implies that the complex
$\eE_{\bullet}$ 
given by (\ref{E:qis}) is of rank zero, 
hence any Schur complex ${\bf S}_{\lambda}(\eE_{\bullet})$
is of rank zero. 
Since 
$\oO_M^{\rm{vir}}$ is 
written as $[Q]-[Q']$ for zero dimensional sheaves
$Q, Q'$, 
we obtain the desired 
identity (\ref{id:vir0})
by Theorem~\ref{thm:formula}.  
\end{proof}

\subsection{NC virtual structure sheaves associated to perfect obstruction theory}
The result of Theorem~\ref{thm:formula}
indicates that one may define the NC virtual structure sheaf
from the perfect obstruction theory, 
without using a quasi NCDG structure.
Let $M$ be a
scheme and 
\begin{align}\label{perfect2}
\phi \colon 
\eE_{\bullet} \to \tau_{\ge -1}\dL_M
\end{align}
a perfect obstruction theory in the sense of~\cite{BF}, i.e. 
$\eE_{\bullet}$ is a two term complex of vector bundles on $M$, 
$\tau_{\ge -1}\dL_M$ is the truncated cotangent complex
of $M$, 
and $\phi$ is the 
morphism in the derived 
category such that 
 $\hH_0(\phi)$ is an isomorphism 
and $\hH_{-1}(\phi)$ is surjective. 
Similarly to (\ref{Kth:vir}), one can define the 
virtual structure sheaf $\oO_M^{\rm{vir}} \in K_0(M)$
using data (\ref{perfect2}) as pointed out in~\cite[Remark~5.4]{BF}.  
The result of Theorem~\ref{thm:formula} naturally leads to the following 
definition: 
\begin{defi}\label{defi:NCvir:E}
For a perfect obstruction theory (\ref{perfect2}) on a scheme $M$, 
the $d$-th NC virtual structure sheaf is defined by 
\begin{align}\label{ncvir:formula22}
(\oO_{M}^{\rm{ncvir}})^{\le d}=\oO_M^{\rm{vir}} \otimes_{\oO_M}
[S_{\oO_M}L_{\oO_M}^{+}(\eE_{\bullet})^{\le d}_{\bullet}]. 
\end{align}
\end{defi}

\begin{exam}
Suppose that $M$ is non-singular, 
hence $\dL_M=\Omega_M$, and 
the perfect obstruction theory (\ref{perfect2}) is given by 
the identity
$\eE_{\bullet}=\Omega_M \to \Omega_M$. 
Then we have $\oO_M^{\rm{vir}}=\oO_M$ and
\begin{align}\label{smooth:K}
(\oO_{M}^{\rm{ncvir}})^{\le d}=
[S_{\oO_M}L_{\oO_M}^{+}(\Omega_M)^{\le d}]. 
\end{align}
If $M$ admits a $d$-smooth NC thickening 
$(M, \oO_M^{\le d})$, then the RHS of (\ref{smooth:K})
coincides with the K-theory class of $\oO_M^{\le d}$. 
\end{exam}

\begin{exam}
Suppose that 
$M^{\rm{vir}}$ has virtual dimension zero, or 
equivalently $\eE_{\bullet}$ is rank zero. 
Similarly to Corollary~\ref{cor:vir0}, 
we have the identity
\begin{align*}
(\oO_{M}^{\rm{ncvir}})^{\le d}=\oO_M^{\rm{vir}}. 
\end{align*}
In particular
if (\ref{perfect2})
is a symmetric perfect obstruction theory (cf.~\cite{BBr}), 
then 
$\eE_{\bullet}$ is rank zero and 
$(\oO_{M}^{\rm{ncvir}})^{\le d}$
coincides with the commutative virtual structure sheaf. 
\end{exam}

\begin{exam}
The definition of (\ref{ncvir:formula22})
also makes sense in the equivariant situation, and 
gives non-trivial examples of 
NC virtual structure sheaves
with virtual dimension zero. 
Let $T=(\mathbb{C}^{\ast})^3$
acts on $\mathbb{C}^3$
by weight $(1, 1, 1)$. 
By regrading $\mathbb{C}^3$ as 
the moduli space of skyscraper sheaves
$\oO_x$ for $x \in \mathbb{C}^3$, we have 
the $T$-equivariant 
perfect obstruction theory 
\begin{align*}
\Omega_{\mathbb{C}^3} \oplus \bigwedge^2 \Omega_{\mathbb{C}^3}[1] 
\to \Omega_{\mathbb{C}^3}. 
\end{align*} 
Let $(t_1, t_2, t_3)$ be the $T$-equivariant parameters. 
By localization, we obtain
the identity in $K_{T}(\mathbb{C}^3)$
\begin{align*}
\oO_{\mathbb{C}^3}^{\rm{vir}}=
\frac{(1-t_1^{-1}t_2^{-1})(1-t_1^{-1}t_3^{-1})(1-t_2^{-1}t_3^{-1})}{(1-t_1^{-1})(1-t_2^{-1})(1-t_3^{-1})}. 
\end{align*}
By Corollary~\ref{cor:virform} (i), 
we have the identities in $K_{T}(\mathbb{C}^3)$
\begin{align*}
&(\oO_{\mathbb{C}^3}^{\rm{ncvir}})^{\le 1} \\
&=\oO_{\mathbb{C}^3}^{\rm{vir}}
\otimes_{\oO_{\mathbb{C}^3}}
\left(1+\bigwedge^2 \Omega_{\mathbb{C}^3} - \Omega_{\mathbb{C}^3} 
\otimes_{\oO_{\mathbb{C}^3}} T_{\mathbb{C}^3} + 
S_{\oO_{\mathbb{C}^3}}^2(T_{\mathbb{C}^3}) \right) \\
&=\frac{(1-t_1^{-1}t_2^{-1})(1-t_1^{-1}t_3^{-1})(1-t_2^{-1}t_3^{-1})}{(1-t_1^{-1})(1-t_2^{-1})(1-t_3^{-1})}
\cdot 
\left(-2+t_1^{-1}t_2^{-1}+t_1^{-1}t_3^{-1}+t_2^{-1}t_3^{-1}
-t_1^{-1}t_2 \right. \\
&\left. \hspace{15mm}
-t_2^{-1}t_3-t_1^{-1}t_3 -t_1 t_2^{-1} -t_2 t_3^{-1} -t_1 t_3^{-1} 
+t_1^2 +t_2^2 +t_3^2+t_1 t_2+t_1 t_3 +t_2 t_3   \right). 
\end{align*}
\end{exam}

\section{Constructions of quasi NCDG structures}\label{sec:const}
In the previous section, we introduced the notion of 
NC virtual structure sheaves (cf.~Definition~\ref{defi:ncvir})
of a quasi NC structure, 
using the notion of a quasi NCDG structure. 
Although NC virtual structure sheaves 
turned out 
to be described 
using the perfect obstruction theory (cf.~Theorem~\ref{thm:formula}), 
still the validity 
of Definition~\ref{defi:ncvir}
relies on the existence
of a quasi NCDG structure. 
In this section, we 
show that the moduli spaces
of graded modules over graded 
algebras admit 
quasi NCDG structures. 
The results in this section will be
used 
in the next section
to show a similar result in a geometric context.

\subsection{Graded algebras and quivers}
Let 
$A$ be a graded algebra
\begin{align}\label{galg}
A=\bigoplus_{i\ge 0}A_i
\end{align}
such that $A_0=\mathbb{C}$ and each $A_i$ is finite dimensional. 
We denote 
by 
\begin{align*}
\mathfrak{m} \cneq A_{>0} \subset A
\end{align*}
the maximal ideal of $A$. 
Let $A \modu_{\gr}$
be the category of finitely generated 
graded left $A$-modules. 
For $M \in A \modu_{\rm{gr}}$, we denote by $M_i$ the 
degree $i$-part of $M$, and write $\lvert a \rvert=i$ 
for non-zero $a \in M_i$. 
For $q>p>0$, we define
\begin{align}\label{def:Apq}
A \modu_{[p, q]} \subset A \modu_{\rm{gr}}
\end{align}
to be
the subcategory of $M \in A \modu_{\rm{gr}}$ 
with $M_i=0$ for $i\notin [p, q]$. 
The category (\ref{def:Apq}) is 
also interpreted as
the category of representations of 
some quiver, defined as follows: 
\begin{defi}\label{defi:Q}
For $q>p>0$, 
 the quiver $Q_{[p, q]}$
is defined as follows: 
the set of vertices
is
\begin{align*}
Q_0=
\{p, p+1, \cdots, q\}. 
\end{align*}
The number of arrows 
in $Q_{[p, q]}$ from $i$ to $j$ is given by   
$\dim_{\mathbb{C}}\mathfrak{m}_{j-i}$.  
The set of arrows in $Q_{[p, q]}$ is denoted by $Q_1$. 
\end{defi}
Below we fix bases of $\mathfrak{m}_k$
for each $k\in \mathbb{Z}_{\ge 1}$, 
and identify the set of arrows from $i$ to $j$ with 
the set of 
basis elements of $\mathfrak{m}_{j-i}$.
Let $\mathbb{C}[Q_{[p, q]}]$
be the path algebra of $Q_{[p, q]}$.  
The multiplication
\begin{align}
\vartheta \colon 
\mathfrak{m}_{j-i} \otimes \mathfrak{m}_{k-j}
 \to \mathfrak{m}_{k-i}
\end{align}
in $A$
defines the relation in $Q_{[p, q]}$, 
by defining the two sided 
ideal 
\begin{align*}
I \subset \mathbb{C}[Q_{[p, q]}]
\end{align*}
to be generated by elements of the form
\begin{align*}
\vartheta(\alpha \otimes \beta)-\alpha \cdot \beta, \ 
\alpha \in \mathfrak{m}_{j-i}, \ \beta \in \mathfrak{m}_{k-j}.
\end{align*}
Here we have regarded $\alpha$, $\beta$ as 
formal linear combinations of paths from $i$ to
$j$, $j$ to $k$, respectively 
and $\alpha \cdot \beta$ is the multiplication in 
$\mathbb{C}[Q_{[p, q]}]$. 
\begin{defi}
We define $\mathrm{Rep}(Q_{[p, q]})$ 
to 
be the category of representations of $Q_{[p, q]}$, 
i.e. 
its objects consist of 
collections
\begin{align}\label{collect}
W=
(\{W_k\}_{k=p}^{q}, \{\phi_a\}_{a \in Q_1}), 
\ \phi_a \colon W_{t(a)} \to W_{h(a)}
\end{align}
where $W_k$ is a finite dimensional vector space
and
$\phi_a$ is a linear map 
for each $a \in Q_1$. 
Here $t(a)$ is the tail of $a$, and $h(a)$ 
is the head of $a$. 
\end{defi}
For a collection (\ref{collect}), 
its \textit{dimension vector} is defined by
\begin{align*}
\dim W \cneq (\dim W_p, \dim W_{p+1}, \cdots, \dim W_q) \in 
\mathbb{Z}_{\ge 0}^{q-p+1}.
\end{align*}
Given a collection (\ref{collect}), 
there is the natural map
\begin{align}\label{nmap}
\mathbb{C}[Q_{[p, q]}] \to \End (W_{\bullet}), \ 
W_{\bullet} \cneq \bigoplus_{k=p}^{q} W_k
\end{align}
sending $a\in Q_1$ to $\phi_a$. 
\begin{defi}
The subcategory 
of $(Q_{[p, q]}, I)$-representations
\begin{align*}
\mathrm{Rep}(Q_{[p, q]}, I) \subset \mathrm{Rep}(Q_{[p, q]})
\end{align*}
is defined 
to be the category of collections (\ref{collect})
such that the map (\ref{nmap}) is zero on $I$. 
\end{defi}
By the construction, sending a collection (\ref{collect}) 
to $W_{\bullet}$ gives the equivalence
\begin{align}\label{equiv}
\mathrm{Rep}(Q_{[p, q]}, I) \stackrel{\sim}{\to}
A \modu_{[p, q]} . 
\end{align}

\subsection{Constructions of commutative dg-schemes}
We are going construct to quasi NCDG structures on 
the moduli spaces of
representations of $Q_{[p, q]}$. 
Before that, following~\cite{BFHR}, 
we recall the constructions of 
smooth commutative dg-structures on smooth 
schemes using the notion of 
curved DGLA. 
In this subsection, we assume that 
$N$ is a commutative smooth scheme. 
\begin{defi}\emph{(\cite[Definition~1.2]{BFHR})}
A bundle of 
curved DGLA over $N$
is a graded 
vector bundle $\lL_{\bullet}$ on $N$, 
endowed with the following data: 
\begin{align*}
\mu \in \Gamma(\lL_2), \ \delta \colon \lL_{\bullet} \to \lL_{\bullet}, \ 
[-, -] \colon \wedge^2 \lL_{\bullet} \to \lL_{\bullet}
\end{align*}
where $\delta$ is an $\oO_N$-module homomorphism 
of degree one (called twisted differential), 
$[-, -]$ is an $\oO_N$-linear super alternating bracket of degree zero, 
which subject to the following axioms: 
\begin{itemize}
\item $\delta(\mu)=0$ as an element of $\Gamma(\lL_3)$. 
\item $\delta \circ \delta=[\mu, -]$. 
\item $\delta$ is a super derivation with respect to the 
bracket $[-, -]$. 
\item The bracket $[-, -]$ satisfies the super Jacobi identity. 
\end{itemize}
\end{defi}
Given a curved DGLA $\lL_{\bullet}$ over $N$, 
we associate a sheaf of super
commutative dg-algebras 
whose underlying $\oO_N$-algebra is 
\begin{align*}
\oO_{N, \bullet} \cneq 
S_{\oO_N}\left(\lL_{\bullet}[1]^{\vee}\right).
\end{align*}
By the Leibniz rule, the differential 
on $\oO_{N, \bullet}$
is determined by 
its restriction to 
$\lL_{\bullet}[1]^{\vee}$
\begin{align*}
q=q_0+q_1+q_2 \colon \lL_{\bullet}[1]^{\vee}  \to 
\oO_N \oplus \lL_{\bullet}[1]^{\vee} \oplus 
S_{\oO_N}^2 \left(\lL_{\bullet}[1]^{\vee}\right)
\end{align*}
where $q_0$ is given by $\mu$, $q_1$ is given by $\delta$ and $q_2$ is given by $[-, -]$. The 
axiom of the curved DGLA shows that $q^2=0$, hence 
we obtain the sheaf of super commutative dg-algebras
$(\oO_{N, \bullet}, q)$ on $N$.

We construct a curved DGLA on $N$
using a graded vector bundle 
\begin{align*}
\vV_{\bullet} \to N
\end{align*}
together with the graded algebra (\ref{galg}). 
Note that $\eE nd_{\oO_N}(\vV_{\bullet})$ is a 
graded vector bundle on $N$
whose 
degree $i$ piece consists of morphisms 
$\vV_{\bullet} \to \vV_{\bullet}$ sending $\vV_j$ to $\vV_{j+i}$. 
We define
\begin{align}\notag
\lL_{n} \cneq \Hom_{\rm{gr}}(\mathfrak{m}^{\otimes n}, 
\eE nd_{\oO_N}(\vV_{\bullet})), \ \lL_{\bullet} 
\cneq \bigoplus_{n>0} \lL_n. 
\end{align}
Here for graded vector spaces 
$W_1, W_2$, we denote by 
$\Hom_{\rm{gr}}(W_1, W_2)$ the 
space of linear maps $W_1 \to W_2$ preserving 
the degrees. 
For example, 
$\lL_1$ is written as
\begin{align*}
\lL_1=\bigoplus_{i\ge 1, j\in \mathbb{Z}}
\hH om_{\oO_N}(\mathfrak{m}_i \otimes \vV_j , \vV_{j+i}). 
\end{align*}
Note that 
$\lL_{\bullet}$ is a
graded vector bundle on $N$. 
We see that $\lL_{\bullet}$ is a sheaf of 
dg-algebras on $N$. 
The differential 
$d \colon \lL_n \to \lL_{n+1}$ is given by
\begin{align*}
df (a_1 \otimes \cdots \otimes a_{n+1})
=\sum_{i=1}^{n} (-1)^{n-i} f(a_1 \otimes \cdots \otimes a_i a_{i+1} \otimes 
\cdots \otimes a_{n+1}). 
\end{align*}
The composition
$\circ  \colon \lL_m \times \lL_n \to \lL_{m+n}$
is given by 
\begin{align*}
f \circ f'(a_1 \otimes \cdots \otimes a_{m+n})
=(-1)^{mn} f(a_1 \otimes \cdots \otimes a_m) \circ 
f'(a_{m+1} \otimes \cdots \otimes a_{m+n}). 
\end{align*}
It is easy to check that the triple
\begin{align}\label{triple2}
(\lL_{\bullet}, d, \circ)
\end{align}
determines the sheaf of dg-algebras on $N$. 

Now suppose that $e$ is a section of $\lL_1 \to N$, 
i.e. $e$ is a degree preserving linear map 
\begin{align*}
e \colon 
\mathfrak{m} \to \End_{\oO_N}(\vV_{\bullet}). 
\end{align*}
We will construct a curved DGLA associated to the above data, with the 
underlying graded vector bundle is
\begin{align*}
\lL_{\ge 2} \cneq \bigoplus_{n\ge 2} \lL_n. 
\end{align*}
The element $\mu \in \Gamma(\lL_2)$ is defined by
\begin{align*}
\mu \cneq de +e \circ e, \ \mu(a_1 \otimes a_2)=e(a_1 a_2) -e(a_1)
 \circ e(a_2).
\end{align*}
The bracket is 
\begin{align*}
[f, f']= f \circ f' -(-1)^{mn} f' \circ f
\end{align*}
for $f \in \lL_n$, $f' \in \lL_m$. The twisted 
differential is defined by
\begin{align*}
\delta \cneq d +[e, -] \colon \lL_{\ge 2} \to \lL_{\ge 2}. 
\end{align*}
It is easy to see
that 
the triple $(\lL_{\ge 2}, \delta, [-, -])$ is 
a curved DGLA, hence determines the 
commutative dg-scheme
\begin{align}\label{construct:dg}
(N, S_{\oO_N}(\lL_{\ge 2}[1]^{\vee}))
=\left(N, S_{\oO_N}\left(\bigoplus_{n\ge 2}
\Hom_{\rm{gr}}\left(\mathfrak{m}^{\otimes n}, \eE nd_{\oO_N}(\vV_{\bullet})
 \right)[1]^{\vee}  \right)  \right). 
\end{align}
The zero-th truncation of the above dg-scheme 
is the scheme theoretic zero locus of the section
$\mu$. 

\subsection{(DG) moduli spaces of graded modules}
Let us fix $q>p>0$ and 
non-negative integers
\begin{align*}
\gamma=(\gamma_p, \gamma_{p+1}, \cdots, \gamma_q)
\in \mathbb{Z}^{q-p+1}_{\ge 0}.  
\end{align*}
Let $W_{\bullet}$ be a finite dimensional graded vector space
written as
\begin{align*}
W_{\bullet}=\bigoplus_{k=p}^{q} W_k, \quad \dim W_k=\gamma_k. 
\end{align*}
Then $W_{\bullet}$ is a graded vector bundle on a point, 
hence the 
construction of the previous subsection 
yields the dg-algebra 
\begin{align*}
L \cneq 
\bigoplus_{n> 0}L_n, \ 
L_n \cneq  \Hom_{\rm{gr}}(\mathfrak{m}^{\otimes n}, \End(W_{\bullet})). 
\end{align*}
We define the following scheme 
theoretic Mauer-Cartan locus
\begin{align}\label{def:Mau}
MC(L) \cneq \{ x \in L^1 : dx + x \circ x=0\}. 
\end{align}
Note that an element 
\begin{align*}
x \in L^1=\Hom_{\rm{gr}}(\mathfrak{m}, \End(W_{\bullet}))
\end{align*}
corresponds 
to a representation of $Q_{[p, q]}$, 
and it 
is contained in $MC(L)$ 
if and only if 
it corresponds to an object in the 
subcategory $\mathrm{Rep}(Q_{[p, q]}, I) \subset \mathrm{Rep}(Q_{[p, q]})$. 
We next 
consider the stability condition on 
$\mathrm{Rep}(Q_{[p, q]})$. 
\begin{defi}
An object $W \in \mathrm{Rep}(Q_{[p, q]})$
is called (semi)stable if for any 
sub
object $0\neq W' \subsetneq W$
in $\mathrm{Rep}(Q_{[p, q]})$, we have the inequality
\begin{align*}
\dim W_p \cdot \dim W'_{q}> (\ge ) \dim W_q \cdot \dim W'_{p}. 
\end{align*} 
\end{defi}
We have the Cartesian square
\begin{align}\label{dia:MC}
\xymatrix{
MC(L)^{s} \ar@{^{(}->}[r] \ar@{^{(}->}[d] & (L^1)^{s} \ar@{^{(}->}[d] \\
MC(L) \ar@{^{(}->}[r] & L^1. 
}
\end{align}
Here $MC(L)^s$, $(L^1)^s$ correspond to 
stable objects in $\mathrm{Rep}(Q_{[p, q]}, I)$, 
$\mathrm{Rep}(Q_{[p, q]})$ respectively. 
The vertical inclusions in (\ref{dia:MC})
are open immersions and 
the horizontal inclusions are closed embeddings. 
Let $G$ be the group of 
degree preserving linear isomorphisms 
$W_{\bullet} \to W_{\bullet}$, i.e. 
\begin{align*}
G \cneq \prod_{k=p}^{q} \GL(W_k). 
\end{align*}
Then $L$ admits the action of $G$ 
by 
\begin{align*}
(g\cdot f)(a_1 \otimes \cdots \otimes a_n)=
g \circ f(a_1 \otimes \cdots \otimes a_n) \circ g^{-1}
\end{align*}
where $f\in L_n$ and $g \in G$.
The dg-algebra structure on 
$L$ is $G$-equivariant, hence 
the diagram (\ref{dia:MC}) is also $G$-equivariant. 
Since the automorphisms 
of stable representations 
are $\mathbb{C}^{\ast}$, the 
stabilizer 
group of the $G$-action on $(L^1)^s$
is 
the diagonal subgroup 
$\mathbb{C}^{\ast} \subset G$, hence 
the $G$-action on $(L^1)^s$ descends to 
the free action of the 
quotient group 
$\overline{G}=G/\mathbb{C}^{\ast}$. 
The free quotients
\begin{align}\label{scheme:s}
M_{\gamma} \cneq MC^{s}(L)/\overline{G},\  N_{\gamma} \cneq 
(L^{1})^{s}/\overline{G}
\end{align}
are indeed obtained as GIT quotients (cf.~\cite{Kin}), hence 
they are quasi projective schemes.
By~\cite{Kin}, 
the scheme $N_{\gamma}$
is the coarse moduli space of 
stable $Q_{[p, q]}$-representations, 
and $M_{\gamma}$ is the closed subscheme of $N_{\gamma}$
corresponding to stable $(Q_{[p, q]}, I)$-representations. 
Note that $N_{\gamma}$ is non-singular, since
it is a free quotient of a smooth variety.  

Now suppose that $\gamma$ is a primitive dimension vector, 
i.e. 
\begin{align*}
\mathrm{g.c.d.}\{ \gamma_i : p\le i\le q\} =1. 
\end{align*}
Then by~\cite{Kin}, $N_{\gamma}$ admits 
a universal representation
\begin{align}\label{univ:rep}
\vV=\left( \{\vV_i\}_{i=p}^{q}, \{\phi_a\}_{a \in Q_1}  \right), \ 
\ \phi_a \colon \vV_{t(a)} \to \vV_{h(a)}
\end{align}
i.e. each $\vV_i$ is a vector 
bundle on $N_{\gamma}$, 
$\phi_a$ is a morphism of vector bundles 
such that for any $x \in N_{\gamma}$, the restriction 
$\vV|_{x}$ is the representation of $Q_{[p, q]}$ corresponding to $x$. 
Note that 
\begin{align*}
\vV_{\bullet}=\bigoplus_{i=p}^{q}\vV_i \to N_{\gamma}
\end{align*}
is a graded vector bundle, and the
collection of  
morphisms $\phi_a$ corresponds to the
graded preserving linear map
\begin{align}\label{map:e}
e \colon \mathfrak{m} \to \End_{\oO_{N_{\gamma}}}(\vV_{\bullet}). 
\end{align}
Therefore 
the construction of the dg-scheme (\ref{construct:dg})
yields the commutative dg-structure on $N_{\gamma}$
\begin{align}\label{com:dga}
(N_{\gamma}, \oO_{N_{\gamma}, \bullet}), \ 
\oO_{N_{\gamma}, \bullet}=S_{\oO_{N_{\gamma}}} \left(\bigoplus_{n\ge 2}
\hH om_{\rm{gr}}
(\mathfrak{m}^{\otimes n} \otimes \vV_{\bullet}, \vV_{\bullet})[1]^{\vee}
\right).
\end{align}
By (\ref{def:Mau}) and 
the construction of $M_{\gamma}$, the zero-th truncation of 
$(N_{\gamma}, \oO_{N_{\gamma}, \bullet})$ 
coincides with the 
closed subscheme $M_{\gamma} \subset N_{\gamma}$.

\subsection{Quasi NC structures on $M_{\gamma}$}\label{subsec:Quasi}
As before, we assume that $\gamma$ is a primitive dimension 
vector
of $Q_{[p, q]}$, 
so that there exists a universal 
$Q_{[p, q]}$-representation (\ref{univ:rep}). 
Let $U \subset N_{\gamma}$ be an affine open subset
such that each $\vV_k|_{U}$ is a trivial 
vector bundle 
\begin{align*}
\vV_k=\oO_U \otimes W_k, \ p\le k\le q
\end{align*}
where $W_k$ is a vector space with dimension $\gamma_k$. 
Since $U$ is a smooth affine scheme, 
there is an NC smooth thickening (cf.~\cite[Theorem~1.6.1]{Kap17})
\begin{align*}
U^{\rm{nc}}=(U, \oO_{U}^{\rm{nc}})
\end{align*}
 on $U$, which is unique up to non-canonical isomorphisms. 
We set
\begin{align*}
\vV_{U, k}^{\rm{nc}} \cneq \oO_U^{\rm{nc}} \otimes W_k, \ 
p\le k\le q
\end{align*}
and regard them as left $\oO_{U}^{\rm{nc}}$-modules. 
Since $\oO_{U}^{\rm{nc}} \twoheadrightarrow \oO_U$ is surjective, 
each morphism $\phi_a$ lifts to a 
left $\oO_{U}^{\rm{nc}}$-module homomorphism
\begin{align}\label{phi:nc}
\phi_a^{\rm{nc}} \colon \vV_{t(a)}^{\rm{nc}}
\to \vV_{h(a)}^{\rm{nc}}, \ a\in Q_1. 
\end{align}
Then the data
\begin{align}\label{lift}
\vV^{\n}_U \cneq (\{\vV_{U, k}^{\rm{nc}}\}_{k=p}^{q}, \{\phi_a^{\rm{nc}}
\}_{a\in Q_1} )
\end{align}
is a flat family 
of representations of $Q_{[p, q]}$ over the
NC scheme $U^{\rm{nc}}$. 
Here we refer to~\cite[Definition~3.1]{Todnc} for the definition of 
flat representations of quivers over NC schemes. 
As before, we set
\begin{align*}
\vV_{U, \bullet}^{\n} \cneq \bigoplus_{k=p}^{q}
\vV_{U, k}^{\n}. 
\end{align*}
The
two sided ideal $\jJ_{U, I} \subset \oO_U^{\n}$
is defined
by the image of
\begin{align}\label{ideal:J}
\vV_{U, \bullet}^{\n} \otimes I \otimes (\vV_{U, \bullet}^{\n})^{\vee} 
\to \vV_{U, \bullet}^{\n} \otimes \End_{\oO_{U}^{\n}}(\vV_{U, \bullet}^{\n})
\otimes
(\vV_{U, \bullet}^{\n})^{\vee} \to 
\oO_{U}^{\n}.
\end{align}
Here the first map of (\ref{ideal:J})
is induced by 
\begin{align*}
I \subset 
\mathbb{C}[Q_{[p, q]}] \to \End_{\oO_{U}^{\n}}(\vV_{U, \bullet}^{\n})
\end{align*}
sending $a \in Q_1$ to $\phi_a^{\n}$, 
and the second map 
of (\ref{ideal:J})
is 
given by $x \otimes g \otimes f \mapsto f \circ g(x)$. 
We set
\begin{align}\label{nc:V}
V\cneq M_{\gamma} \cap U, \ 
\oO_V^{\n} \cneq \oO_U^{\n}/\overline{\jJ}_{U, I}, \ 
V^{\n} \cneq (V, \oO_V^{\n}). 
\end{align}
Here $\overline{\jJ}_{U, I}$ is the topological 
closure of $\jJ_{U, I}$ with respect to the NC filtration 
of $\oO_U^{\n}$. 
Then $V^{\n}$ is an NC structure on $V$. 

Note that giving 
a collection of 
morphisms (\ref{phi:nc})
is equivalent to giving an element
\begin{align}\label{hate}
\widehat{e}
 \in \Hom_{\rm{gr}}(\mathfrak{m}, \End_{\oO_{U}^{\n}}(\vV_{U, \bullet}^{\n}))
\end{align}
such that $\widehat{e}^{ab}=e|_{U}$, where 
$e$ is the universal map (\ref{map:e}).  
Similarly to the construction (\ref{triple2}), 
the direct sum
\begin{align*}
\bigoplus_{n> 0} 
\Hom_{\rm{gr}}(\mathfrak{m}^{\otimes n}, \End_{\oO_{U}^{\n}}(\vV_{U, \bullet}^{\n}))
\end{align*}
is a dg-algebra. 
We set 
\begin{align*}
\widehat{\mu} \cneq d\widehat{e} +\widehat{e} \circ \widehat{e} \in
\Hom_{\rm{gr}}(\mathfrak{m}^{\otimes 2}, \End_{\oO_{U}^{\n}}
(\vV_{U, \bullet}^{\n})). 
\end{align*}
Then we have the natural morphism
of $\oO_U^{\n}$ bi-module
\begin{align}\label{map:mu}
\left(\vV_{U, \bullet}^{\n} \otimes \mathfrak{m}^{\otimes 2} 
\otimes (\vV_{U, \bullet}^{\n})^{\vee}  \right)_{0}
\to \oO_U^{\n}. 
\end{align}
Here $(-)_{0}$ means the degree zero part, 
and the map (\ref{map:mu}) is
given by
\begin{align*}
x\otimes 
a_1 \otimes a_2 \otimes f \mapsto 
f \circ \widehat{\mu}(a_1 \otimes a_2)(x).
\end{align*}
From the construction, it is easy 
to see that the image of (\ref{map:mu}) 
coincides with $\jJ_{U, I} \subset \oO_U^{\n}$. 

We can give a moduli theoretic 
interpretation of the NC thickening $V^{\n}$ 
of $V$. 
Let $\nN$ be the category of NC nilpotent algebras and 
\begin{align*}
h_{\gamma}|_{V} \colon \nN \to \sS et
\end{align*}
the functor 
sending $R$ 
to the isomorphism classes of triples $(f, \wW, \psi)$: 
\begin{itemize}
\item $f$ is a morphism of schemes 
$f \colon \Spec R^{ab} \to V$. 
\item $\wW$ is a flat representation of $(Q_{[p, q]}, I)$
over $\Spf R$. 
\item $\psi$ is an isomorphism 
$\psi \colon \wW^{ab} \stackrel{\cong}{\to} f^{\ast}\vV$
as $(Q_{[p, q]}, I)$-representations over $\Spec R^{ab}$. 
\end{itemize}
An isomorphism 
$(f, \wW, \psi) \to (f', \wW', \psi')$ 
exists if $f=f'$, and 
there is an isomorphism $\wW \to \wW'$ 
as representations of $(Q_{[p, q]}, I)$ over $\Spf R$ 
commuting $\psi$, $\psi'$.
\begin{prop}\emph{(\cite[Proposition~3.11]{Todnc})}\label{prop:nchull}
The natural transformation
\begin{align}\label{nat:trans}
h_{V^{\n}} \cneq \Hom(\Spf(-), V^{\n}) \to h_{\gamma}|_{V}
\end{align}
sending $g \colon \Spf R \to V^{\rm{nc}}$ to 
$(g^{ab}, g^{\ast}\vV_U^{\n}, \id)$
is an NC hull of $h_{\gamma}|_{V}$, i.e. 
(\ref{nat:trans}) is an isomorphism 
on the category of commutative algebras, 
and 
for any
central extension (\ref{central})
in $\nN$, 
we have the surjection: 
\begin{align*}
h_{V^{\n}}(R_1) \twoheadrightarrow h_{\gamma}|_{V}(R_1) 
\times_{h_{\gamma}|_{V}(R_2)}
h_{V^{\n}}(R_2).
\end{align*}
\end{prop}
Let $\{U_i\}_{i \in \mathbb{I}}$ be an affine open 
cover of $N_{\gamma}$ such that 
each $\vV_k|_{U_i}$ is trivial, and 
set $V_i \cneq M_{\gamma} \cap U_i$. 
Applying the construction (\ref{nc:V}), we obtain 
affine NC structures on each $V_i$
\begin{align*}
V_i^{\n}=(V_i, \oO_{V_i}^{\n}), \ i \in \mathbb{I}. 
\end{align*}
Using Proposition~\ref{prop:nchull}, we proved the following in~\cite{Todnc}:
\begin{thm}\emph{(\cite[Corollary~3.12]{Todnc})}\label{thm:qnc}
There exist isomorphisms 
\begin{align*}
\phi_{ij} \colon 
V_j^{\n}|_{V_{ij}} \stackrel{\cong}{\to}
V_i^{\n}|_{V_{ij}}, \ g_{ij} \colon 
\phi_{ij}^{\ast}\vV_{U_i}^{\n}|_{V_{ij}} \stackrel{\cong}{\to}
\vV_{U_j}^{\n}|_{V_{ij}}
\end{align*}
where $\phi_{ij}$ are isomorphisms of NC schemes
giving a quasi NC structure 
on $M_{\gamma}$, 
and $g_{ij}$ are isomorphisms 
of representations of $(Q_{[p, q]}, I)$
over $V_j^{\n}|_{V_{ij}}$. 
\end{thm}
The constructions of this subsection and 
the previous subsection are summarized by the following diagram: 
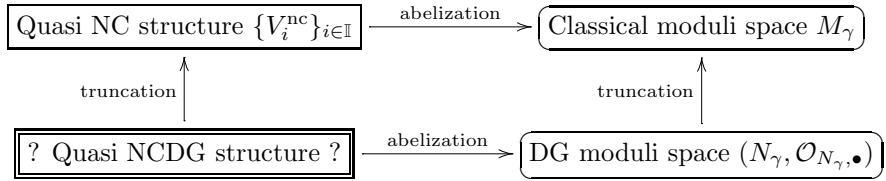
\begin{align}\notag
\xymatrix{
\fbox{\mbox{Quasi NC structure} $\{V_i^{\n}\}_{i\in \mathbb{I}}$} \ar[rr]^-{\rm{abelization}} & &
\ovalbox{Classical moduli space $M_{\gamma}$} \\
\doublebox{? Quasi NCDG structure ?} \ar[u]^{\rm{truncation}} 
\ar[rr]^-{\rm{abelization}} & &
\ovalbox{\mbox{DG moduli space} $(N_{\gamma}, \oO_{N_{\gamma}, \bullet})$} 
\ar[u]^{\rm{truncation}} 
}
\end{align}
Below, we
are going to construct a quasi NCDG structure 
which fits into the above diagram.

\subsection{Constructions of non-commutative dg-algebras}\label{subsec:ndga}
Let $A$ be
a graded algebra (\ref{galg}), 
and $R$
an another associative (not necessary commutative)
algebra. 
Let $P$ be a graded free right $R$-module, and 
set $P^{\vee} \cneq \Hom_{R}(P, R)$
which 
is a graded free left $R$-module.
We set
\begin{align*}
\mathfrak{\mathfrak{P}} \cneq \bigoplus_{n\ge 2} \left(P^{\vee} \otimes 
\mathfrak{m}^{\otimes n} 
\otimes P\right)_{0}
\end{align*}
  which is a free graded $R$ bi-module. 
Here 
the grading of $\mathfrak{P}$ on 
$(P^{\vee} \otimes \mathfrak{m}^{\otimes n}
 \otimes P)_{0}$
is set to be $1-n$. 
We define the graded algebra $\mathfrak{A}$ to be
the tensor algebra of $\mathfrak{P}$ over $R$
\begin{align*}
\mathfrak{A} \cneq 
\bigoplus_{m\ge 0}
\mathfrak{P}^{\otimes_{R} m}, \ 
\mathfrak{P}^{\otimes_{R} m} \cneq 
 \overset{m}{\overbrace{\mathfrak{P} \otimes_{R} \mathfrak{P} \otimes_{R}
 \cdots \otimes_{R} \mathfrak{P}}}. 
\end{align*}
The grading on $\mathfrak{A}$ is induced by that of $\mathfrak{P}$, 
and 
the degree zero part of $\mathfrak{A}$ is 
$\mathfrak{P}^{\otimes_{R} 0} \cneq R$. 
The algebra structure on $\mathfrak{A}$ is 
given by 
\begin{align*}
(b_1 \otimes \cdots \otimes b_m) 
\cdot (b_{m+1} \otimes \cdots \otimes b_n)
=b_1 \otimes \cdots \otimes 
b_{m} \otimes b_{m+1} \otimes \cdots \otimes b_n. 
\end{align*}
Let
\begin{align}\label{egrade}
\widehat{e} \colon {\mathfrak{m}} \to \End_{R}(P)
\end{align}
be a grade preserving linear map. 
For $a \in \mathfrak{m}$, 
$x \in P$ and 
$f \in P^{\vee}$, we set
\begin{align}\label{convent}
ax \cneq \widehat{e}(a)(x) \in P, \ 
fa \cneq f \circ \widehat{e}(a) \in P^{\vee}. 
\end{align}
We also set the linear map
\begin{align*}
\widehat{\mu} \colon \mathfrak{m}^{\otimes 2} \to \End_{R}(P)
\end{align*}
as follows: 
\begin{align*}
\widehat{\mu}(a_1 \otimes a_2)=\widehat{e}(a_1 a_2)-\widehat{e}(a_1) 
\circ \widehat{e}(a_2). 
\end{align*}
We define the degree one ${R}$ bi-module map 
\begin{align}\label{map:Q}
Q=Q_0+Q_1+Q_2 \colon 
\mathfrak{P} \to {R} \oplus \mathfrak{P} \oplus (\mathfrak{P}\otimes_{R} \mathfrak{P})
\end{align}
in the following way.
The map $Q_0$ is defined by
\begin{align}
\label{def:q0}
Q_0 \colon \left(P^{\vee} \otimes \mathfrak{m}^{\otimes 2} \otimes P \right)_0
 &\to {R} \\
\notag
f \otimes a_1 \otimes a_2 \otimes x &\mapsto f \circ 
\widehat{\mu}(a_1, a_2)(x). 
\end{align}
The map $Q_1$ is defined by
\begin{align*}
Q_1 \colon \left(P^{\vee} \otimes \mathfrak{m}^{\otimes n} \otimes P\right)_{0}
& \to 
\left(P^{\vee} \otimes \mathfrak{m}^{\otimes n-1} \otimes P\right)_{0} \\
f \otimes a_1 \otimes \cdots \otimes a_n \otimes x
 &\mapsto 
(-1)^{n+1}fa_1 \otimes a_2 \otimes \cdots \otimes a_{n} \otimes x \\
&\quad +\sum_{j=1}^{n-1} (-1)^{n+1-j} f \otimes a_1
 \otimes \cdots \otimes a_j a_{j+1}
\otimes \cdots \otimes a_{n} \otimes x \\
&\hspace{40mm} -f \otimes a_1 \otimes \cdots \otimes a_{n-1} \otimes a_n x. 
\end{align*}
Here we have used the convention in (\ref{convent}).  
Finally the map $Q_2$ is defined by
\begin{align*}
Q_2 \colon \left(P^{\vee} \otimes 
\mathfrak{m}^{\otimes n} \otimes P\right)_{0} &\to 
\bigoplus_{k=2}^{n-2}
\left(P^{\vee} \otimes \mathfrak{m}^{\otimes k} \otimes P \right)_0
 \otimes_{R} \left( P^{\vee}
\otimes \mathfrak{m}^{\otimes n-k} \otimes P \right)_0
\\
f \otimes a_1 \otimes \cdots \otimes a_n \otimes x
 &\mapsto \sum_{k=2}^{n-2} (-1)^{n(k-2)+1} f
\otimes a_1 \otimes \cdots \otimes a_k \otimes \widehat{\id}_P \\
&\hspace{40mm} \otimes a_{k+1} \otimes \cdots \otimes a_{n} \otimes x.
\end{align*}
Here
$\widehat{\id}_P$
is defined as follows: 
we decompose $\id_P \in \Hom_{R}(P, P)=P \otimes_{R} P^{\vee}$
as 
\begin{align*}
\id_P=\sum_{i} u_i \otimes _{R} v_i
\end{align*}
for  
homogeneous elements $u_i \in P$, $v_i \in P^{\vee}$ with 
$\lvert u_i \rvert + \lvert v_i \rvert =0$, and 
\begin{align*}
\widehat{\id}_P \cneq
\sum_{\lvert f \rvert + \lvert a_1 \rvert+\cdots +
\lvert a_k \rvert+ \lvert u_i \rvert=0}
u_i \otimes_{R} v_i. 
\end{align*} 
By the Leibniz rule, the map (\ref{map:Q}) extends to the 
degree one ${R}$ bi-module map
\begin{align}\label{map:QB}
Q \colon \mathfrak{A} \to \mathfrak{A}. 
\end{align}
We have the following proposition: 
\begin{prop}\label{prop:Q}
The map $Q$ in (\ref{map:QB}) 
satisfies $Q^2=0$. 
Hence $(\mathfrak{A}, Q)$ is a non-commutative 
differential graded algebra. 
\end{prop}
\begin{proof}
It 
is straightforward to check $Q^2=0$, 
and we leave the details to the readers. 
\end{proof}
The first few terms of the 
complex $(\mathfrak{A}, Q)$ is
\begin{align*}
\cdots \to (P^{\vee} \otimes
\mathfrak{m}^{\otimes 3} \otimes P)_{0} 
&\oplus \left(P^{\vee} \otimes
\mathfrak{m}^{\otimes 2} 
\otimes P \right)^{\otimes_{R} 2}_0 \\
&\to (P^{\vee} \otimes
\mathfrak{m}^{\otimes 2} \otimes P)_{0}
\stackrel{Q_0}{\to} {R} \to 0. 
\end{align*}
In particular, we have
\begin{align}\label{isom:h0}
\hH_0(\mathfrak{A}, Q)={R}/J
\end{align}
where $J$ is the two sided ideal 
given by the image of $Q_0$ in (\ref{def:q0}).

\subsection{Abelization of $\mathfrak{A}$}
We describe the abelization of the 
non-commutative dg-algebra $\mathfrak{A}$. 
We set
\begin{align*}
\mathfrak{P}^{ab} &\cneq \bigoplus_{n\ge 2} \left(\mathfrak{m}^{\otimes n} \otimes
\End_{{R}^{ab}}(P^{ab}) \right)_0 \\
& =\bigoplus_{n\ge 2} 
\Hom_{\rm{gr}}(\mathfrak{m}^{\otimes n}, \End_{{R}^{ab}}(P^{ab}))^{\vee}. 
\end{align*}
which is a graded free ${R}^{ab}$-module. 
The grading on $(\mathfrak{m}^{\otimes n} \otimes \End_{{R}^{ab}}(P^{ab}))_0$
is $1-n$, and $\ast^{\vee}$ is the dual of $\ast$ over ${R}^{ab}$. 
\begin{lem}\label{lem:abel}
As a graded algebra, we have 
\begin{align}\label{Bab}
\mathfrak{A}^{ab}=S_{{R}^{ab}}(\mathfrak{P}^{ab}). 
\end{align}
\end{lem}
\begin{proof}
We write $P$ as 
$P=W \otimes{R}$
for a graded vector space $W$, and set
\begin{align*}
\overline{W} \cneq 
\bigoplus_{n\ge 2} \left(W^{\vee} \otimes \mathfrak{m}^{\otimes n} 
\otimes W\right)_0. 
\end{align*}
Then we have 
$\mathfrak{P}={R} \otimes \overline{W} \otimes {R}$, and 
\begin{align*}
\mathfrak{A}={R} \ast T(\overline{W}). 
\end{align*}
On the other hand, we have 
$\mathfrak{P}^{ab}={R}^{ab} \otimes \overline{W}$, hence
\begin{align*}
S_{{R}^{ab}}(\mathfrak{P}^{ab})={R}^{ab} \otimes
S(\overline{W}). 
\end{align*}
Therefore we obtain (\ref{Bab}). 
\end{proof}
By the Leibniz rule, the derivation (\ref{map:QB})
induces a degree one ${R}^{ab}$-linear derivation
\begin{align*}
q \cneq Q^{ab} \colon \mathfrak{A}^{ab} \to \mathfrak{A}^{ab}. 
\end{align*}
From the description of $Q$, 
it is easy to describe $q$ 
under the identity (\ref{Bab}). 
By Lemma~\ref{lem:abel}, the 
map $q$ is determined by
its restriction to $\mathfrak{P}^{ab}$
\begin{align*}
q=
q_0+q_1+q_2 \colon 
\mathfrak{P}^{ab} \to {R}^{ab} \oplus \mathfrak{P}^{ab} \oplus S^2_{{R}^{ab}}(\mathfrak{P}^{ab}). 
\end{align*}
Let $e$, $\mu$ be the compositions of $\widehat{e}$, $\widehat{\mu}$,
with
the natural map $\End_{R}(P) \to \End_{{R}^{ab}}(P^{ab})$:  
\begin{align}\label{data:eu}
e \colon \mathfrak{m} \to \End_{{R}^{ab}}(P^{ab}), \ 
\mu \colon \mathfrak{m}^{\otimes 2} \to \End_{{R}^{ab}}(P^{ab}). 
\end{align}
The map $q_0$ is described as 
\begin{align*}
q_0 \colon \left(\mathfrak{m}^{\otimes 2} \otimes \End_{{R}^{ab}}(P^{ab})\right)_0
&\to {R}^{ab} \\
a_1 \otimes a_2 \otimes g &\mapsto
\tr(g \circ \mu(a_1, a_2)).
\end{align*}
The map $q_1$ is described as 
\begin{align*}
q_1 \colon \left(\mathfrak{m}^{\otimes n} \otimes \End_{{R}^{ab}}(P^{ab}) 
\right)_{0}
&\to \left(\mathfrak{m}^{\otimes n-1} \otimes \End_{{R}^{ab}}(P^{ab})
\right)_{0} \\
a_1 \otimes \cdots \otimes a_n \otimes g  &\mapsto 
(-1)^{n+1}a_2 \otimes \cdots \otimes a_n \otimes (g \circ e(a_1)) \\
&\quad
+\sum_{j=1}^{n-1} (-1)^{n+1-j} a_1 \otimes \cdots \otimes a_j a_{j+1} \otimes
\cdots \otimes a_n \otimes g  \\
&\hspace{30mm}
-a_1 \otimes \cdots \otimes a_{n-1} \otimes (e(a_n) \circ g). 
\end{align*}
The map $q_2$ is described as 
\begin{align*}
&q_2 \colon 
\left(\mathfrak{m}^{\otimes n} \otimes \End_{{R}^{ab}}(P^{ab})\right)_0 \\
&\hspace{20mm} \to 
\bigoplus_{k=2}^{n-2}
\left(\mathfrak{m}^{\otimes k} \otimes 
\End_{{R}^{ab}}(P^{ab}) \right)_0 \otimes_{{R}^{ab}}
\left( \End_{{R}^{ab}}(P^{ab}) \otimes 
\mathfrak{m}^{\otimes n-k} \right)_0 \\
&a_1 \otimes \cdots \otimes a_n \otimes g \\
& \hspace{20mm} \mapsto 
\sum_{k=2}^{n-2}
(-1)^{n(k-2)+1}
a_1 \otimes \cdots \otimes a_k \otimes 
\widehat{\circ}^{\vee} g \otimes 
a_{k+1} \otimes \cdots \otimes a_n. 
\end{align*}
Here 
$\circ^{\vee}$ is the dual of 
the composition map
\begin{align*}
\circ^{\vee} \colon 
\End_{{R}^{ab}}(P^{ab}) \to \End_{{R}^{ab}}(P^{ab}) \otimes_{{R}^{ab}}
\End_{{R}^{ab}}(P^{ab})
\end{align*}
and writing $\circ^{\vee} g$ as 
the sum of $u_i \otimes_{{R}^{ab}} v_i$
for homogeneous elements 
$u_i, v_i \in \End_{{R}^{ab}}(P^{ab})$, 
we set
\begin{align*}
\widehat{\circ}^{\vee} g
=\sum_{\lvert a_1 \rvert + \cdots + \lvert a_k \rvert + \lvert u_i \rvert=0}
u_i \otimes_{{R}^{ab}} v_i. 
\end{align*}
On the other hand, note that
\begin{align*}
\widetilde{P}^{ab} \to \Spec {R}^{ab}
\end{align*}
is a graded vector bundle
on $\Spec R^{ab}$.  
The data of $e$ in (\ref{data:eu})
together with
the construction of (\ref{construct:dg})
yield the affine commutative dg-scheme
\begin{align}\label{dg:affine}
\left(\Spec R^{ab}, S_{\widetilde{R}^{ab}}\left( \bigoplus_{n\ge 2}
\Hom_{\rm{gr}}
(\mathfrak{m}^{\otimes n}, \eE nd_{\widetilde{R}^{ab}}
(\widetilde{P}^{ab}))[1]^{\vee}  \right) \right). 
\end{align} 
By Lemma~\ref{lem:abel}
together with the above description of $q=Q^{ab}$, 
the global section 
of the dg-structure sheaf of (\ref{dg:affine})
coincides with 
$\mathfrak{A}^{ab}$ as a dg-algebra.

\subsection{Quasi NCDG structures on $N_{\gamma}$}
Now we return to the situation of Subsection~\ref{subsec:Quasi}. 
As in Subsection~\ref{subsec:Quasi}, we 
take an
affine open subset 
$U \subset N_{\gamma}$
such that each $\vV_k|_{U}$
is trivial $\vV_k|_{U}=\oO_{U} \otimes W_k$. 
We take an NC smooth thickening $U^{\n}$ of $U$, 
and a lift 
$\vV_{U, \bullet}^{\n}$ of 
$\vV|_{U}$
to a flat representation of $Q_{[p, q]}$ 
over $U^{\n}$, as in (\ref{lift}). 
We apply the construction
in Subsection~\ref{subsec:ndga}
by setting
\begin{align*}
R=\Gamma(\oO_U^{\n}), \ 
P=\Gamma(\vV_{U, \bullet}^{\n})^{\vee}
\end{align*} 
where $\ast^{\vee}$ is the dual of $\ast$
over $R$. 
Note that $P$ is a graded free right $R$-module. 
Using (\ref{hate}) instead of (\ref{egrade}), 
the construction in Proposition~\ref{prop:Q}
yields the non-commutative dg-algebra 
structure on 
\begin{align}\label{dga:lambda}
\Lambda_{U, \bullet}^{\n} &\cneq 
\bigoplus_{m\ge 0} 
\left(\bigoplus_{n\ge 2}
\Gamma(\vV_{U, \bullet}^{\n}) \otimes
\mathfrak{m}^{\otimes n} 
\otimes\Gamma(\vV_{U, \bullet}^{\n})^{\vee}
  \right)^{\otimes_{\oO_U^{\n}}m}_0.
\end{align}
By the proof of Lemma~\ref{lem:abel},
we have 
\begin{align}\label{U:ncdg}
\Lambda_{U, \bullet}^{\n}=
\Gamma(\oO_U^{\n}) \ast T(\overline{W})
\end{align}
where 
$\overline{W}$ is the finite dimensional graded
vector space given by
\begin{align*}
\overline{W}=\bigoplus_{n\ge 2}
\left(\bigoplus_{p\le j, k \le q}W_j \otimes \mathfrak{m}^{\otimes n}
\otimes W_k^{\vee}   \right)_0.
\end{align*}
Here $W_j$ is a vector space with dimension $\gamma_j$, 
located in degree $j$. 
We define the following
affine NCDG scheme 
\begin{align}\label{SpfRUn}
(U, \oO_{U, \bullet}^{\n}) \cneq \Spf \Lambda_{U, \bullet}^{\n}. 
\end{align}
Note that (\ref{SpfRUn}) is smooth 
as $U^{\n}$ is a smooth NC thickening of $U$. 
By the identity (\ref{isom:h0}), we have 
\begin{align}\label{id:V}
\tau_0(U, \oO_{U, \bullet}^{\n})=
(V, \oO_V^{\n})
\end{align}
where $\oO_V^{\n}$ 
is given in (\ref{nc:V}), 
as it is given by the NC completion 
of the cokernel of (\ref{map:mu}). 
Also the argument in the previous subsection 
shows that
\begin{align}\label{Uab}
(\oO_{U, \bullet}^{\n})^{ab} = 
\oO_{N_{\gamma}, \bullet}|_{U}
\end{align}
where $\oO_{N_{\gamma}, \bullet}$ is the 
sheaf of commutative dg-algebras on $N_{\gamma}$
given in (\ref{com:dga}). 
Hence (\ref{SpfRUn}) is 
an affine NCDG structure on $(U, \oO_{N_{\gamma}, \bullet}|_{U})$. 

Let $\{U_i\}_{i\in \mathbb{I}}$ be an affine open 
cover of $N_{\gamma}$, 
such that each $\vV_k|_{U_i}$ is trivial. 
Applying the above construction, we obtain 
affine NCDG schemes
\begin{align*}
(U_i, \oO_{U_i, \bullet}^{\n}), \ i \in \mathbb{I}. 
\end{align*}
On the other hand, we have isomorphisms 
of NC schemes and $Q_{[p, q]}$-representations
over $U_j^{\n}|_{U_{ij}}$
(cf.~\cite[Corollary~3.12]{Todnc})
\begin{align}\label{isom:rep}
\phi_{ij} \colon U_j^{\n}|_{U_{ij}} \stackrel{\cong}{\to}
U_i^{\n}|_{U_{ij}}, \ 
g_{ij} \colon \phi_{ij}^{\ast}\vV_{U_i}^{\n}|_{U_{ij}} \stackrel{\cong}{\to}
\vV_{U_j}^{\n}|_{U_{ij}}
\end{align}
such that $\phi_{ij}^{ab}=\id$
and 
$g_{ij}^{ab}$ 
is the 
gluing isomorphism 
of the universal object $\vV$
in (\ref{univ:rep}). 
Since the dg-algebra (\ref{dga:lambda})
is determined by the 
algebra $\Gamma(\oO_U^{\n})$ together 
with the $Q_{[p, q]}$-representation
$\vV_{U}^{\n}$ over $U^{\n}$, 
the isomorphisms (\ref{isom:rep})
 induce the isomorphisms
of NCDG schemes
\begin{align*}
\phi_{ij, \bullet} \colon 
(U_{ij}, \oO_{U_j, \bullet}^{\n}|_{U_{ij}}) \stackrel{\cong}{\to}
(U_{ij}, \oO_{U_i, \bullet}^{\n}|_{U_{ij}})
\end{align*}
giving a quasi NCDG structure on $(N_{\gamma}, \oO_{N_{\gamma}, \bullet})$. 
Also under (\ref{id:V}), 
the isomorphisms $\hH_0(\phi_{ij, \bullet})$
give a quasi NC structure on $M_{\gamma}$
considered in Theorem~\ref{thm:qnc}. 
As a summary, we have obtained the following: 
\begin{thm}\label{thm:ncdg}
There exists a smooth quasi NCDG 
structure on the smooth 
commutative dg-moduli space 
$(N_{\gamma}, \oO_{N_{\gamma}, \bullet})$
whose 
zero-th 
truncation 
gives a quasi NC structure on $M_{\gamma}$
in Theorem~\ref{thm:qnc}. 
\end{thm}

\section{Quasi NCDG structures on the moduli spaces of stable sheaves}
\label{sec:quasi}
In this section, we 
show that quasi NC structures on the moduli spaces of 
stable sheaves
on projective schemes constructed in~\cite{Todnc}
are 
obtained as the zero-th truncations of 
smooth quasi NCDG structures on 
smooth commutative dg-moduli spaces of stable sheaves. 
Throughout this section, 
we assume that 
$(X, \oO_X(1))$ is a connected polarized
projective scheme over $\mathbb{C}$. 
\subsection{Moduli spaces of stable sheaves}
For $F \in \Coh(X)$,
let $\alpha(F, t)$ be its Hilbert polynomial
\begin{align*}
\alpha(F, t) \cneq \chi(F \otimes \oO_X(t))
\end{align*}
and $\overline{\alpha}(F, t) \cneq \alpha(F, t)/c$ 
its reduced Hilbert polynomial, where 
$c$ is the leading coefficient of $\alpha(F, t)$.  
Recall the (semi)stability on $X$:
\begin{defi}
A
coherent sheaf $F$ on $X$
is called (semi)stable if 
it is a pure sheaf, 
and for any 
subsheaf $0 \subsetneq F' \subsetneq F$, we have 
\begin{align}\label{def:stab}
\overline{\alpha}(F', k) <(\le) \overline{\alpha}(F, k), \ 
k\gg 0.   
\end{align}
\end{defi}
Let us take a polynomial $\alpha \in \mathbb{Q}[t]$, 
which is a Hilbert polynomial of some coherent sheaf on $X$.  
Let 
\begin{align}\label{mfunct}
\mM_{\alpha} \colon 
\sS ch/\mathbb{C} \to \sS et
\end{align}
be the functor 
defined by 
\begin{align*}
\mM_{\alpha}(T) \cneq \left\{ \fF \in \Coh(X \times T) :
\begin{array}{c}
\fF \mbox{ is } T \mbox{-flat, } 
\fF_t \mbox{ for any }
t\in T \mbox{ is }\\
\mbox{ stable with Hilbert polynomial }
\alpha 
\end{array} \right\}/(\mbox{equiv}). 
\end{align*}
Here  
$\fF$ and $\fF'$
are \textit{equivalent} if there is an line bundle $\lL$ on $T$
such that $\fF \cong \fF' \otimes p_T^{\ast}\lL$, 
where $p_T \colon X \times T \to T$ is the projection. 
The moduli functor (\ref{mfunct}) is not always
representable by a scheme, but 
if we assume that 
\begin{align}\label{primitive}
\mathrm{g. c. d.}\{\alpha(m) : m\in \mathbb{Z}\}=1
\end{align}
then (\ref{mfunct})
is represented by a projective scheme $M_{\alpha}$
(cf.~\cite{Mu2}), i.e.
there is an isomorphism of functors
\begin{align}\label{funct:isom}
\Hom(-, M_{\alpha}) \stackrel{\cong}{\to}
\mM_{\alpha}. 
\end{align}
We call $\alpha$ satisfying the condition 
(\ref{primitive}) as \textit{primitive}.
Below, we always assume
that $\alpha$ is primitive. 
Note that the isomorphism (\ref{funct:isom}) is
induced by a universal family 
\begin{align*}
\uU \in \Coh(X \times M_{\alpha}). 
\end{align*}
\subsection{DG moduli spaces of stable sheaves}
Note that $X=\mathrm{Proj}(A)$ for 
the graded algebra
\begin{align*}
A=\bigoplus_{i\ge 0}H^0(X, \oO_X(i)). 
\end{align*}
We use the
quiver with 
relation $(Q_{[p, q]}, I)$ 
constructed in Definition~\ref{defi:Q}
from the above graded algebra $A$. 
For $q > p >0$, we set
\begin{align}\label{GammaU}
\Gamma_{[p, q]}(\uU) 
\cneq \bigoplus_{i=p}^{q}
p_{M\ast}(\uU \otimes p_X^{\ast}\oO_X(i)). 
\end{align}
Here
$p_M$, $p_X$ are the projections from 
$X \times M_{\alpha}$ to $M_{\alpha}$, $X$
respectively. 
If we take $q \gg p \gg 0$, then 
(\ref{GammaU}) is a flat representation of 
$(Q_{[p, q]}, I)$
over $M_{\alpha}$
with the primitive dimension vector 
\begin{align}\label{gam:prim}
\gamma=(\alpha(p), \alpha(p+1), \cdots, \alpha(q)). 
\end{align}
The object (\ref{GammaU})
defines the morphism of schemes
\begin{align}\label{Upsilon}
\Upsilon \colon 
M_{\alpha} \to M_{\gamma}
\end{align}
which is an open immersion by~\cite[Corollary~3.4]{BFHR}. 
Let $M_{[p, q]} \subset M_{\gamma}$ be the image of $\Upsilon$. 
Since both sides of (\ref{Upsilon}) are projective, the image 
$M_{[p, q]}$ 
is a union of connected components of $M_{\gamma}$. 
We have the isomorphism of schemes
\begin{align}\label{Upsilon:isom}
\Upsilon \colon M_{\alpha} \stackrel{\cong}{\to}
M_{[p, q]}. 
\end{align}
In particular, the object (\ref{GammaU})
is a universal $(Q_{[p, q]}, I)$-representation 
restricted to $M_{[p, q]}$. 
By replacing $\uU$ by $\uU \otimes p_M^{\ast}\lL$
for some line bundle $\lL$ on $M_{\alpha}$, 
we may assume that 
the universal family $\vV_{\bullet}$ given in (\ref{univ:rep})
restricted to $M_{[p, q]}$ 
coincides with (\ref{GammaU}).  

Recall that $M_{\gamma}$ is obtained as the 
zero-th truncation of the 
smooth commutative dg-scheme 
$(N_{\gamma}, \oO_{N_{\gamma}, \bullet})$
given by (\ref{com:dga}). 
By taking a suitable open 
subset $N_{\alpha} \subset N_{\gamma}$ 
containing $M_{[p, q]}$,  
the following result was proved in~\cite{BFHR}: 
\begin{thm}\emph{(\cite{BFHR})}\label{thm:smdg}
There is an open subset 
$N_{\alpha} \subset N_{\gamma}$ such that the 
 smooth commutative dg-scheme
\begin{align}\label{com:dga2}
(N_{\alpha}, \oO_{N_{\alpha}, \bullet} \cneq 
\oO_{N_{\gamma}, \bullet}|_{N_{\alpha}})
\end{align}
satisfies the following: 
\begin{itemize}
\item The zero-th truncation of (\ref{com:dga2}) 
is isomorphic to $M_{\alpha}$. 
\item For any $[E] \in M_{\alpha}$, the tangent 
complex of (\ref{com:dga2}) at $[E]$ is 
\begin{align}\label{tan:com}
\mathrm{Cone}(\mathbb{C} \to \dR \Hom(E, E))[1]. 
\end{align}
\end{itemize}
\end{thm}
\subsection{Existence of a quasi NCDG structure}
One can also extend the isomorphism (\ref{Upsilon:isom}) to their 
NC thickenings. 
Let
\begin{align*}
h_{\alpha} \colon \nN \to \sS et
\end{align*}
be the functor sending $R\in \nN$ 
to the isomorphism classes of 
triples $(f, \fF, \psi)$: 
\begin{itemize}
\item $f$ is a morphism of schemes 
$f \colon \Spec R^{ab} \to M_{\alpha}$. 
\item $\fF$ is an object of $\Coh(X_{R})$
which is flat over $R$, where $X_{R}\cneq X \times \Spf R$. 
\item $\psi$ is an isomorphism $\psi \colon \fF^{ab} 
\stackrel{\cong}{\to} f^{\ast}\uU$. 
\end{itemize}
An isomorphism $(f, \fF, \psi) \to (f', \fF', \psi')$
exists if $f=f'$, 
and there is an isomorphism $\fF \to \fF'$ in $\Coh(X_{R})$
commuting $\psi$, $\psi'$. 

\begin{prop}\emph{(\cite[Proposition~4.13]{Todnc})}\label{thm:extend}
The isomorphism (\ref{Upsilon:isom})
extends to the isomorphism of functors
\begin{align*}
\Gamma_{[p, q]} \colon 
h_{\alpha} \stackrel{\cong}{\to} h_{\gamma}|_{M_{[p, q]}}. 
\end{align*}
\end{prop}
By combining Theorem~\ref{thm:ncdg}, Theorem~\ref{thm:smdg} and 
Proposition~\ref{thm:extend}, 
we obtain the following: 
\begin{thm}\label{thm:ncvir2}
There is a smooth quasi NCDG structure 
$\{(U_i, \oO_{U_i, \bullet}^{\n})\}_{i\in \mathbb{I}}$
on the smooth commutative dg-moduli space
$(N_{\alpha}, \oO_{N_{\alpha}, \bullet})$
such that 
the zero-th truncations
\begin{align}\label{qnc:V}
\{(V_i, \oO_{V_i}^{\n})\}_{i\in \mathbb{I}} \cneq
\{\tau_0(U_i, \oO_{U_i, \bullet}^{\n})\}_{i \in \mathbb{I}}
\end{align}
is a quasi NC structure on $M_{\alpha}$
which fit into NC hulls $h_{V_i^{\n}} \to h_{\alpha}|_{V_i}$. 
\end{thm}
The quasi NC 
structure in (\ref{qnc:V}) 
is the one constructed in~\cite{Todnc}. 
By~\cite[Theorem~1.2]{Todnc}, 
it satisfies the following
condition. 
For $[E] \in V_i$, let 
$\widehat{\oO}_{V_i, [E]}^{\n}$
be the completion of $\oO_{V_i}^{\n}$ at $[E]$. 
Then it coincides with the 
pro-representable hull of the 
NC deformation functor of $E$
developed in~\cite{Lau}, \cite{Erik}, \cite{ESe}, \cite{ELO}, 
\cite{ELO2}, \cite{ELO3}.
This implies that we have an isomorphism of algebras
\begin{align*}
\widehat{\oO}_{V_i, [E]} \cong
R_E^{\n}
\end{align*}
where $R_E^{\n}$ is the algebra (\ref{intro:R})
constructed by the $A_{\infty}$-structure
(\ref{intro:A}). 

\subsection{An example}
We take $X=\mathbb{P}^2$ and 
$\alpha$ to be the constant function $1$. 
Note that a stable sheaf on $X$ has Hilbert polynomial $1$
if and only if it is a skyscraper sheaf $\oO_x$
for $x\in \mathbb{P}^2$. 
Then the moduli space $M_{\alpha}$ is isomorphic to 
$\mathbb{P}^2$ itself. 
On the other hand, by Beilinson's theorem~\cite{Bei},
we have the derived equivalence
\begin{align*}
\dR \Hom(\eE, -) \colon D^b(\Coh(\mathbb{P}^2)) \stackrel{\sim}{\to}
D^b (\modu A)
\end{align*} 
where $\eE$ and $A$ are given by
\begin{align*}
\eE=\oO_{\mathbb{P}^2} \oplus \oO_{\mathbb{P}^2}(-1) \oplus 
\oO_{\mathbb{P}^2}(-2), \ A=\End(\eE).
\end{align*}
By the above equivalence, one can take $p=0$ and $q=2$ in the 
argument of the previous subsection. 
The quiver $Q_{[0, 2]}$
is described as
\begin{align}\label{fig:Q}
\xymatrix{ \stackrel{0}{\bullet} 
\ar@<2ex>[r]_{x_1}
\ar@<0ex>[r]_{x_2}
\ar@<-2ex>[r]_{x_3}
\ar@/^/@<3ex>[rr]_{z_{11}}
\ar@/^/@<5ex>[rr]_{z_{22}}
\ar@/^/@<7ex>[rr]_{z_{33}}
\ar@/_/@<-3ex>[rr]^{z_{12}}
\ar@/_/@<-5ex>[rr]^{z_{13}}
\ar@/_/@<-7ex>[rr]^{z_{23}}
& \stackrel{1}{\bullet}
\ar@<2ex>[r]_{y_1}
\ar@<0ex>[r]_{y_2}
\ar@<-2ex>[r]_{y_3}
& \stackrel{2}{\bullet}
 }
\end{align}
with relations 
given by
\begin{align*}
z_{ij}=y_{j} x_i=y_i x_j, \ 1\le i \le j \le 3. 
\end{align*}
The dimension vector
(\ref{gam:prim}) is $\gamma=(1, 1, 1)$. 
The moduli space $N_{\gamma}$ of representations of 
$Q_{[0, 2]}$ without relation
is the quotient of the stable locus of 
$\mathbb{C}^3 \times \mathbb{C}^3 \times \mathbb{C}^6$
by $(\mathbb{C}^{\ast})^{\times 2}$. 
It contains an open subset $U \subset N_{\gamma}$ which 
parametrizes representations of (\ref{fig:Q})
with $x_3=y_3=1$ and 
$x_1, x_2, y_1, y_2, z_{ij} \in \mathbb{C}$, i.e.
\begin{align*}
U=\Spec \mathbb{C}[x_1, x_2, y_1, y_2, z_{ij} : 1\le i \le j \le 3]. 
\end{align*}
An NC smooth thickening of $U$ is given by
\begin{align*}
U^{\n}=\Spf R, \ R=\mathbb{C}\langle 
x_1, x_2, y_1, y_2, z_{ij} : 1\le i \le j \le 3 
\rangle_{[\hspace{-0.5mm}[ab]\hspace{-0.5mm}]}. 
\end{align*}
Let $\vV$ be the universal representation 
of $Q_{[0, 2]}$ on $N_{\gamma}$, which is a rank three 
vector bundle. 
A lift of $\vV|_{U}$ to $U^{\n}$ is given by the 
following representation
\begin{align}\notag
\xymatrix{ R
\ar@<2ex>[r]_{\cdot x_1}
\ar@<0ex>[r]_{\cdot x_2}
\ar@<-2ex>[r]_{1}
\ar@/^/@<3ex>[rr]_{\cdot z_{11}}
\ar@/^/@<5ex>[rr]_{\cdot z_{22}}
\ar@/^/@<7ex>[rr]_{\cdot z_{33}}
\ar@/_/@<-3ex>[rr]^{\cdot z_{12}}
\ar@/_/@<-5ex>[rr]^{\cdot z_{13}}
\ar@/_/@<-7ex>[rr]^{\cdot z_{23}}
& R
\ar@<2ex>[r]_{\cdot y_1}
\ar@<0ex>[r]_{\cdot y_2}
\ar@<-2ex>[r]_{1}
& R.
 }
\end{align}
The algebra (\ref{U:ncdg}) is then given by 
\begin{align*}
\Lambda_{U, \bullet}^{\n}=R \ast T(\mathfrak{m}_1^{\otimes 2})
\end{align*}
where $\mathfrak{m}_1=H^0(\mathbb{P}^2, \oO_{\mathbb{P}^2}(1))$
and $\mathfrak{m}_1^{\otimes 2}$
is located in degree $-1$. 
It is written as
a suitable NC completion of
\begin{align*}
\mathbb{C}\langle x_1, x_2, y_1, y_2, z_{ij}, w_{kl} : 
1\le i \le j \le 3, 1\le k, l \le 3 \rangle
\end{align*} 
where $\deg x_i=\deg y_i=\deg z_{ij}=0$ and $\deg w_{kl}=-1$ and the 
differential is given by
\begin{align*}
Q \colon 
w_{kl} \mapsto z_{kl}-y_k x_l
\end{align*}
where we set $z_{kl}=z_{lk}$ if $k>l$ and 
$x_3=y_3=1$. 
Then 
$\Spf \Lambda_{U, \bullet}^{\n}$ is
a smooth affine NCDG structure on 
its abelization
$(U, \oO_U \otimes S(\mathfrak{m}_1^{\otimes 2}))$. 
The  zero-th 
truncation 
of $\Spf \Lambda_{U, \bullet}^{\n}$
gives $\mathbb{C}^2 \subset \mathbb{P}^2=M_{\alpha\equiv 1}$. 

\subsection{NC virtual structure sheaves on moduli spaces of stable sheaves}
Now we assume that 
the smooth 
commutative dg-scheme (\ref{com:dga2}) is a $[0, 1]$-manifold, 
which means that
the tangent complex of (\ref{com:dga2}) has 
amplitude in $[0, 1]$. 
By (\ref{tan:com}), 
this is equivalent to the condition 
\begin{align}\label{high:ob}
\Ext^{i}(E, E)=0, \ i\ge 3
\end{align}
for any $[E] \in M_{\alpha}$. 
Applying Definition~\ref{defi:ncvir}
to the quasi NCDG structure in Theorem~\ref{thm:ncvir2}, 
we obtain the 
$d$-th NC virtual structure sheaf
\begin{align*}
(\oO_{M_{\alpha}}^{\rm{ncvir}})^{\le d} \in K_0(M_{\alpha}). 
\end{align*}
By Theorem~\ref{thm:formula}
and (\ref{tan:com}), 
we have the following:
\begin{cor}
The $d$-th NC virtual structure sheaf 
associated to the quasi NCDG structure in 
Theorem~\ref{thm:ncvir2} is written as 
\begin{align}\label{ncvir:formula2}
(\oO_{M_{\alpha}}^{\rm{ncvir}})^{\le d}=\oO_{M_{\alpha}}^{\rm{vir}} 
\otimes_{\oO_{M_{\alpha}}}
[S_{\oO_{M_{\alpha}}}L_{\oO_{M_{\alpha}}}^{+}
(\eE_{\bullet})^{\le d}_{\bullet}]. 
\end{align}
Here $\eE_{\bullet} \to \tau_{\ge -1} \dL_{M_{\alpha}}$
is a perfect obstruction theory on $M_{\alpha}$
such that
for any $[E] \in M_{\alpha}$ we have
\begin{align*}
\hH_0(\eE_{\bullet}^{\vee}|_{[E]})=\Ext^1(E, E), \ 
\hH_1(\eE_{\bullet}^{\vee}|_{[E]})=\Ext^2(E, E). 
\end{align*}
\end{cor}

\providecommand{\bysame}{\leavevmode\hbox to3em{\hrulefill}\thinspace}
\providecommand{\MR}{\relax\ifhmode\unskip\space\fi MR }
\providecommand{\MRhref}[2]{%
  \href{http://www.ams.org/mathscinet-getitem?mr=#1}{#2}
}
\providecommand{\href}[2]{#2}


Kavli Institute for the Physics and 
Mathematics of the Universe, University of Tokyo,
5-1-5 Kashiwanoha, Kashiwa, 277-8583, Japan.

\textit{E-mail address}: yukinobu.toda@ipmu.jp

\end{document}